\documentclass[11pt]{amsart}
\usepackage{amscd,amsmath,amssymb,amsfonts,verbatim}
\usepackage[cmtip, all]{xy}
\usepackage{comment}

\newtheorem{thm}{Theorem}[section]
\newtheorem{prop}[thm]{Proposition}
\newtheorem{lem}[thm]{Lemma}
\newtheorem{cor}[thm]{Corollary}

\theoremstyle{definition}

\newtheorem{defn}[thm]{Definition}
\theoremstyle{remark}
\newtheorem{remk}[thm]{Remark}
\newtheorem{remks}[thm]{Remarks}

\newtheorem{exm}[thm]{Example}
\newtheorem{exms}[thm]{Examples}
\newtheorem{notat}[thm]{Notation}
\numberwithin{equation}{section}

{\hfill$\square$\end{defn}}
{\hfill$\square$\end{remk}}
{\hfill$\square$\end{remks}}
{\hfill$\square$\end{exm}}
{\hfill$\square$\end{exms}}
{\hfill$\square$\end{notat}}

\newcommand{\thmref}{Theorem~\ref}
\newcommand{\propref}{Proposition~\ref}
\newcommand{\corref}{Corollary~\ref}

\newcommand{\lemref}{Lemma~\ref}

\newcommand{\sE}{{\mathcal E}}
\newcommand{\sF}{{\mathcal F}}

\newcommand{\sH}{{\mathcal H}}
\newcommand{\sI}{{\mathcal I}}

\newcommand{\sK}{{\mathcal K}}
\newcommand{\sL}{{\mathcal L}}

\newcommand{\sO}{{\mathcal O}}

\newcommand{\sR}{{\mathcal R}}

\newcommand{\sU}{{\mathcal U}}

\newcommand{\sW}{{\mathcal W}}

\newcommand{\sZ}{{\mathcal Z}}

\newcommand{\F}{{\mathbb F}}
\newcommand{\G}{{\mathbb G}}
\renewcommand{\H}{{\mathbb H}}

\newcommand{\N}{{\mathbb N}}
\renewcommand{\P}{{\mathbb P}}
\newcommand{\Q}{{\mathbb Q}}

\newcommand{\Z}{{\mathbb Z}}

\newcommand{\fp}{{\mathfrak p}}
\newcommand{\fq}{{\mathfrak q}}

\newcommand{\Alb}{{\rm Alb}}
\newcommand{\CH}{{\rm CH}}

\newcommand{\surj}{\twoheadrightarrow}
\newcommand{\inj}{\hookrightarrow}

\newcommand{\Pic}{{\rm Pic}}

\newcommand{\Spec}{{\rm Spec \,}}

\newcommand{\divf}{{\rm div}}

\newcommand{\Sch}{{\operatorname{\mathbf{Sch}}}}

\newcommand{\Sm}{{\mathbf{Sm}}}

\newcommand{\ds}{{/\kern-3pt/}}

\newcommand{\ov}{\overline}

\renewcommand{\dim}{\text{\rm dim}}

\newcommand{\tuborg}{\left\{\begin{array}{ll}}
\newcommand{\sluttuborg}{\end{array}\right.}

\newcommand{\zar}{{\rm zar}}

\newcommand{\wt}{\widetilde}

\newcommand{\etl}{{\acute{e}t}}

\newcounter{elno}

\newcounter{elno-abc}   
\newenvironment{listabc}{
                         \begin{list}{\alph{elno-abc})
                                     }{\usecounter{elno-abc}}
                      }{
                         \end{list}}
\newcounter{elno-abc-prime}   
\newenvironment{listabcprime}{
                         \begin{list}{\alph{elno-abc-prime}')
                                     }{\usecounter{elno-abc-prime}}
                      }{
                         \end{list}}

\begin{document}
\title{Torsion in the 0-cycle group with modulus}
\author{Amalendu Krishna}
\address{School of Mathematics, Tata Institute of Fundamental Research,  
1 Homi Bhabha Road, Colaba, Mumbai, India}
\email{amal@math.tifr.res.in}





\maketitle

\begin{quote}\emph{Abstract.}  We show, for a smooth projective
variety $X$ over an algebraically closed field $k$ with 
an effective Cartier divisor $D$, that the torsion subgroup $\CH_0(X|D)\{l\}$  
can be described in terms of a relative {\'e}tale cohomology for
any prime $l \neq p = {\rm char}(k)$. This extends
a classical result of Bloch, on the torsion in the ordinary Chow group, to
the modulus setting. 
We prove the Roitman torsion theorem (including $p$-torsion)
for $\CH_0(X|D)$ when $D$ is reduced. We deduce applications to the problem
of invariance of the prime-to-$p$ torsion in $\CH_0(X|D)$ under an infinitesimal
extension of $D$. 
\end{quote}
\setcounter{tocdepth}{1}
\tableofcontents  


\section{Introduction}\label{sec:Intro}
One of the fundamental problems in the study of the theory of motives,
as envisioned by Grothendieck, is to construct a universal cohomology
theory of varieties and describe it in terms of algebraic cycles. 
When $X$ is a smooth quasi-projective scheme over a base field $k$, 
such a motivic cohomology theory of $X$ is known to exist (see, for example,
\cite{Voev-PST}, \cite{Levine-Mixed}). Moreover, 
Bloch \cite{Bloch-M} showed that this cohomology theory has
an explicit description in terms of groups of algebraic cycles, called 
the higher Chow groups. These groups have all 
the properties that one expects, including Chern classes and a 
Chern character isomorphism from the higher $K$-groups of Quillen.

However, the situation becomes much complicated when we go beyond the world of
smooth varieties. There is no theory of motivic cohomology of singular
schemes known till date which could be a universal cohomology theory,
which could be described in terms of algebraic cycles and, which could
be rationally isomorphic to the algebraic $K$-theory. 

With the aim of understanding the algebraic $K$-theory of the ring 
$k[t]/(t^n)$ in terms of algebraic cycles, Bloch and Esnault introduced
the idea of algebraic cycles with modulus, 
called additive Chow groups. This idea was substantially expanded
in the further works of Park \cite{Park}, R\"ulling \cite{Ruelling}
and Krishna-Levine \cite{KLevine}.  

In recent works of Binda-Saito  \cite{BS} and Kerz-Saito \cite{KeS},
a theory of higher Chow groups with modulus was introduced,
which generalizes the construction of the additive higher Chow groups.
These groups, denoted $\CH^*(X|D, *)$,  are designed to 
study the arithmetic and geometric properties of a smooth variety $X$ with 
fixed conditions along an effective  (possibly non-reduced) Cartier divisor 
$D$ on it (see \cite{KeS}), 
and are supposed to give a cycle-theoretic description of the 
relative $K$-groups $K_*(X,D)$, defined as the homotopy groups of 
the homotopy fiber of the restriction map $K(X)\to K(D)$ (see \cite{IK}). 

In order to  provide evidence that the Chow groups with modulus are the 
right motivic cohomology groups to compute the relative $K$-theory of a 
smooth scheme with respect to an effective divisor, one would like to know if 
these groups share enough of the known structural properties of the 
Chow groups without modulus, and, if these groups could be related to
other cohomological invariants of a pair $(X,D)$. 
In particular, one would like to know
if the torsion part of the Chow group of 0-cycles with modulus
could be described in terms of a relative {\'e}tale cohomology or, in terms of
an Albanese variety with modulus.

\subsection{The main results}
For the Chow group of 0-cycles without modulus on a smooth projective scheme, 
Bloch \cite{Bloch-tor} showed that its prime-to-$p$ torsion is completely 
described in terms of an {\'e}tale cohomology. Roitman \cite{Roitman} and
Milne \cite{Milne-1} showed that the torsion in the Chow group 
is completely described in terms of the Albanese of the
underlying scheme.
When $D$ is an effective Cartier divisor on a smooth projective
scheme $X$, the Albanese variety with modulus $A^d(X|D)$ with appropriate 
universal property and the Abel-Jacobi map 
$\rho_{X|D}: \CH_0(X|D)_{\deg 0} \to A^d(X|D)$ were constructed in \cite{BK}.

The goal of this work is to extend the torsion results of
Bloch and Roitman-Milne  
to 0-cycles with modulus. Our main results thus broadly look
as follows. Their precise statements and underlying hypotheses 
and notations will be explained at appropriate places in the text.

\begin{thm}$($Bloch's torsion theorem$)$\label{thm:Intro-1}
Let $k$ be an algebraically closed field of exponential characteristic $p$.
Let $X$ be a smooth projective scheme of dimension $d \ge 1$ over $k$ and let
$D \subset X$ be an effective Cartier divisor. 
Then, for any prime $l \neq p$, there is an isomorphism
\[
\lambda_{X|D}: \CH_0(X|D)\{l\} \xrightarrow{\simeq} 
H^{2d-1}_{\etl}(X|D, {\Q_l}/{\Z_l}(d)).
\]
\end{thm}

\begin{thm}$($Roitman's torsion theorem$)$\label{thm:Intro-2}
Let $k$ be an algebraically closed field of exponential characteristic $p$.
Let $X$ be a smooth projective scheme of dimension $d \ge 1$ over $k$ and let
$D \subset X$ be an effective Cartier divisor. Assume that $D$ is reduced.
Then, the Albanese variety with modulus $A^d(X|D)$ is a semi-abelian variety
and the Abel-Jacobi map $\rho_{X|D}: \CH_0(X|D)_{\deg 0} \to A^d(X|D)$ is an
isomorphism on the torsion subgroups (including $p$-torsion).
\end{thm}

Note that the isomorphism of \thmref{thm:Intro-1} is very subtle because one 
does not have any obvious map in either direction. The construction of such a 
map is an essential part of the problem.
In the case of Chow group 
without modulus, the construction of $\lambda_X$ by Bloch makes essential use 
of the Bloch-Ogus theory and, more importantly, it uses the Weil-conjecture.
The prime-to-$p$ part of \thmref{thm:Intro-2} was earlier proven in \cite{BK}.
The new contribution here is the description of the more challenging $p$-part.

As part of the proof of \thmref{thm:Intro-2}, we prove the Roitman
torsion theorem (including $p$-torsion) for singular separably weakly normal 
surfaces (see Definition~\ref{defn:swn}).
This provides the first evidence that the Roitman torsion theorem 
may be true for non-normal varieties in positive characteristic.

\begin{thm}\label{thm:Intro-3}
Let $X$ be a reduced projective separably weakly normal surface over an
algebraically closed field $k$ of exponential characteristic $p$.
Then, the Albanese variety $A^2(X)$ is a semi-abelian variety
and the Abel-Jacobi map $\rho_{X}: \CH_0(X)_{\deg 0} \to A^2(X)$ is an
isomorphism on the torsion subgroups (including $p$-torsion).
\end{thm}

Note that \thmref{thm:Intro-2} makes no assumption on $\dim(X)$ but this is not
the case for \thmref{thm:Intro-3}. The reason for this is that our proof of
\thmref{thm:Intro-2} uses a weaker version of the Lefschetz type theorem
for separably weakly normal varieties. We do not know if
a general hyperplane section of an arbitrary separably weakly normal variety
is separably weakly normal. This is known to be false for
weakly normal varieties in positive characteristic (see \cite{CGM*}).
We shall investigate this question in a future project.

\vskip .3cm

\subsection{Applications}
As an application of \thmref{thm:Intro-1}, we obtain the following 
result about the prime-to-$p$ torsion in the Chow group with modulus.

\begin{cor}\label{cor:Intro-1*}
Let $k$ be an algebraically closed field of exponential characteristic $p$.
Let $X$ be a smooth projective scheme of dimension $d \ge 1$ over $k$ and
let $D \subset X$ be an effective Cartier divisor.
Then, the restriction map ${_{n}{\CH_0(X|D)}} \to {_{n}{\CH_0(X|D_{\rm red})}}$ 
is an isomorphism for every integer $n$ prime to $p$.
\end{cor}

In \cite[Theorem~1.3]{Miyazaki}, an isomorphism of similar indication
has been very recently
shown after inverting ${\rm char}(k)$. But it has no implication on
\corref{cor:Intro-1*}.
Another application of \thmref{thm:Intro-1} is the following extension of the
results of Bloch \cite[Chap.~5]{Bloch-duke} to
the relative $K$-theory.

\begin{cor}\label{cor:Intro-3}
Let $X$ be a smooth projective surface over an algebraically closed
field $k$ of exponential characteristic $p$ and let $l \neq p$ be a
prime. Let $D \subset X$ be an effective Cartier divisor. Then, the following 
hold.
\begin{enumerate}
\item
$H^1(X, \sK_{2, (X,D)})\otimes_{\Z} {\Q_l}/{\Z_l} = 0$.
\item
$H^2(X, \sK_{2, (X,D)})\{l\} \simeq H^3_{\etl}(X|D, {\Q_l}/{\Z_l}(2))$.
\end{enumerate}
\end{cor}

Analogous results are also obtained for surfaces with arbitrary singularities
(see Theorems~\ref{thm:K-2-coh-0} and ~\ref{thm:Cyc-iso}).

\subsection{Outline of proofs}
Our proofs of the above results are based on the decomposition theorem
of \cite{BK} which relates the Chow group of 0-cycles with modulus on
a smooth variety to the Levine-Weibel Chow group of 0-cycles on a 
singular variety. Using this decomposition theorem, our main task
becomes extending the torsion results of Bloch and Roitman-Milne to
singular varieties.
We prove these results for the Chow group of 0-cycles on singular surfaces
in the first part (\S~\ref{sec:K-2-surf} and \S~\ref{sec:finite-dim})
of this text. These results are of independent interest
in the study of 0-cycles on singular varieties.

In \S~\ref{sec:TEC} and ~\S~\ref{sec:RTM}, we prove our results about 
0-cycles with modulus by using the analogous results for singular surfaces
and an induction on dimension. This induction is based on some variants of 
the Lefschetz hyperplane section theorem for the Albanese variety and the
weak Lefschetz theorem for the {\'e}tale cohomology with modulus. 
Some applications are deduced in \S~\ref{sec:TEC}.

\subsection{Notations}
Throughout this text (except in \S~\ref{sec:WN}), 
$k$ will denote an algebraically
closed field of exponential characteristic $p \ge 1$. We shall denote the
category of separated schemes of finite type over $k$ by $\Sch_k$. 
The subcategory of $\Sch_k$ consisting of smooth schemes
over $k$ will be denoted by $\Sm_k$. 
An undecorated product $X \times Y$ in $\Sch_k$ will mean that it is
taken over the base scheme $\Spec(k)$. An undecorated cohomology group 
will be assumed to be with respect to the Zariski site.
We shall let $w: {\Sch_k}/{\etl} \to {\Sch_k}/{\zar}$
denote the canonical morphism between the {\'e}tale and the Zariski
sites of $\Sch_k$. For $X \in \Sch_k$, the notation $X^N$ will usually mean the
normalization of $X_{\rm red}$ unless we use a different notation in a specific
context.
For any abelian group $A$ and an integer $n \ge 1$, the subgroup of
$n$-torsion elements in $A$ will be denoted by ${_{n}A}$ and, for any prime
$l$, the $l$-primary torsion subgroup will be denoted by
$A\{l\}$.

\section{Review of 0-cycle groups and Albanese varieties}
\label{sec:cyc-alb}
To prove the main results of \S~\ref{sec:Intro},
our strategy is to use the decomposition theorem for 0-cycles from \cite{BK}. 
This theorem asserts that the Chow group of 0-cycles with modulus
on a smooth scheme is a direct summand of the Chow group of 0-cycles 
(in the sense of \cite{LW}) on a singular scheme. We therefore need to prove
analogues of Theorems~\ref{thm:Intro-1} and ~\ref{thm:Intro-2} for
singular schemes. In the first few sections of this paper, our goal is
to prove such results for arbitrary singular surfaces. These
results are of independent interest and answer some questions
in the study of 0-cycles on singular schemes. The higher dimensional cases of
Theorems~\ref{thm:Intro-1} and ~\ref{thm:Intro-2} will be deduced from the
case of surfaces. 

In this section, we review the definitions and some
basic facts about the Chow group of 0-cycles on singular schemes,
the Chow group with modulus and their universal quotients, called the
Albanese varieties.

\subsection{0-cycles on singular schemes}\label{sec:Chow-sing}
We recall the definition of the cohomological Chow group of 0-cycles for
singular schemes from \cite{BK} and \cite{LW}. 
Let $X$ be a reduced quasi-projective
scheme of dimension $d \ge 1$ over $k$. Let $X_{\rm sing}$ and $X_{\rm reg}$
denote the loci of the singular and the regular points of $X$.
Given a nowhere dense closed subscheme $Y \subset X$ such that 
$X_{\rm sing} \subseteq Y$
and no component of $X$ is contained in $Y$, we let $\sZ_0(X,Y)$ denote
the free abelian group on the closed points of $X \setminus Y$.
We write $\sZ_0(X, X_{\rm sing})$ in short as $\sZ_0(X)$.

\begin{defn}\label{defn:0-cycle-S-1}
Let $C$ be a pure dimension one reduced scheme in $\Sch_k$. 
We shall say that a pair $(C, Z)$ is \emph{a good curve
relative to $X$} if there exists a finite morphism $\nu\colon C \to X$
and a  closed proper subscheme $Z \subsetneq C$ such that the following hold.
\begin{enumerate}
\item
No component of $C$ is contained in $Z$.
\item
$\nu^{-1}(X_{\rm sing}) \cup C_{\rm sing}\subseteq Z$.
\item
$\nu$ is local complete intersection at every 
point $x \in C$ such that $\nu(x) \in X_{\rm sing}$. 
\end{enumerate}
\end{defn}

Let $(C, Z)$ be a good curve relative to $X$ and let 
$\{\eta_1, \cdots , \eta_r\}$ be the set of generic points of $C$. 
Let $\sO_{C,Z}$ denote the semilocal ring of $C$ at 
$S = Z \cup \{\eta_1, \cdots , \eta_r\}$.
Let $k(C)$ denote the ring of total
quotients of $C$ and write $\sO_{C,Z}^\times$ for the group of units in 
$\sO_{C,Z}$. Notice that $\sO_{C,Z}$ coincides with $k(C)$ 
if $|Z| = \emptyset$. 
As $C$ is Cohen-Macaulay, $\sO_{C,Z}^\times$  is the subgroup of $k(C)^\times$ 
consisting of those $f$ which are regular and invertible in the local rings 
$\sO_{C,x}$ for every $x\in Z$. 

Given any $f \in \sO^{\times}_{C, Z} \inj k(C)^{\times}$, we denote by  
${\rm div}_C(f)$ (or ${\rm div}(f)$ in short) 
the divisor of zeros and poles of $f$ on $C$, which is defined as follows. If 
$C_1,\ldots, C_r$ are the irreducible components of $C$, 
and $f_i$ is the factor of $f$ in $k(C_i)$, we set 
${\rm div}(f)$ to be the $0$-cycle $\sum_{i=1}^r {\rm div}(f_i)$, where 
${\rm div}(f_i)$ is the usual 
divisor of a rational function on an integral curve in the sense of
\cite{Fulton}. 
As $f$ is an invertible 
regular function on $C$ along $Z$, ${\rm div}(f)\in \sZ_0(C,Z)$.

By definition, given any good curve $(C,Z)$ relative to $X$, we have a 
push-forward map $\sZ_0(C,Z)\xrightarrow{\nu_{*}} \sZ_0(X)$.
We shall write $\sR_0(C, Z, X)$ for the subgroup
of $\sZ_0(X)$ generated by the set 
$\{\nu_*({\rm div}(f))| f \in \sO^{\times}_{C, Z}\}$. 
Let $\sR_0(X)$ denote the subgroup of $\sZ_0(X)$ generated by 
the image of the map $\sZ_0(C, Z, X) \to \sZ_0(X)$, where
$(C, Z)$ runs through all good curves relative to $X$.
We let $\CH_0(X) = \frac{\sZ_0(X)}{\sR_0(X)}$.

If we let $\sR^{LW}_0(X)$ denote the subgroup of $\sZ_0(X)$ generated
by the divisors of rational functions on good curves as above, where
we further assume that the map $\nu: C \to X$ is a closed immersion,
then the resulting quotient group ${\sZ_0(X)}/{\sR^{LW}_0(X)}$ is
denoted by $\CH^{LW}_0(X)$. Such curves on $X$ are called the 
{\sl Cartier curves}. There is a canonical surjection
$\CH^{LW}_0(X) \surj \CH_0(X)$. The Chow group $\CH^{LW}_0(X)$ was
discovered by Levine and Weibel \cite{LW} in an attempt to describe the
Grothendieck group of a singular scheme in terms of algebraic cycles.
The modified version $\CH_0(X)$ was introduced in \cite{BK}.
As an application of our main results, we shall be able to identify
the two Chow groups for certain singular schemes (see \thmref{cor:LW-lci-join}).

\subsection{The Chow group of the double}
\label{sec:double}
Let $X$ be a smooth quasi-projective scheme of dimension $d$ over $k$
and let $D \subset X$ be an effective Cartier divisor. Recall from 
\cite[\S~2.1]{BK} that the double of $X$ along $D$ is a quasi-projective
scheme $S(X,D) = X \amalg_D X$ so that
\begin{equation}\label{eqn:rel-et-2}
\xymatrix@C1pc{
D \ar[r]^-{\iota} \ar[d]_{\iota} & X \ar[d]^{\iota_+} \\
X \ar[r]_-{\iota_-} & S(X,D)}
\end{equation}
is a co-Cartesian square in $\Sch_k$.
In particular, the identity map of $X$ induces a finite map
$\Delta:S(X,D) \to X$ such that $\Delta \circ \iota_\pm = {\rm Id}_X$
and $\pi = \iota_+ \amalg \iota_-: X \amalg X \to S(X,D)$ 
is the normalization map.
We let $X_\pm = \iota_\pm(X) \subset S(X,D)$ denote the two irreducible
components of $S(X,D)$.  We shall often write $S(X,D)$ as $S_X$
when the divisor $D$ is understood. $S_X$ is a reduced quasi-projective
scheme whose singular locus is $D_{\rm red} \subset S_X$. 
It is projective whenever
$X$ is so. The following easy algebraic lemma shows that ~\eqref{eqn:rel-et-2}
is also a Cartesian square.

\begin{lem}\label{lem:rel-et-2*}
Let 
\begin{equation}\label{eqn:rel-et-2*-0}
\xymatrix@C1pc{
R \ar[r]^-{\psi_1} \ar[d]_{\psi_2} & A_1 \ar@{->>}[d]^{\phi_1} \\
A_2 \ar@{->>}[r]^-{\phi_2} & B}
\end{equation}
be a Cartesian square of commutative Noetherian rings 
such that ${\rm Ker}(\phi_i) = (a_i)$ for $i = 1,2$. Then
$A_1 \otimes_R A_2 \xrightarrow{\simeq} B$.
\end{lem}
\begin{proof}
If we let $\alpha_1 = (0, a_2)$ and $\alpha_2 = (a_1, 0)$ in $R \subset
A_1 \times A_2$, then it is easy to check that for $i =1,2$, the map
$\psi_i$ is surjective and ${\rm Ker}(\psi_i) = (\alpha_i)$.
This implies that $A_1 \otimes_{R} A_2 \simeq {R}/{(\alpha_1)} \otimes_R 
A_2 \simeq {A_2}/{(a_2)} = B$.
\end{proof}

We shall later consider a more general situation than
~\eqref{eqn:rel-et-2} where
we allow the two components of the double to be distinct. This
general case is a non-affine version of ~\eqref{eqn:rel-et-2*-0} and 
is used in the proof of \thmref{thm:Intro-2}.

\subsection{The Albanese varieties for singular schemes}
\label{sec:alb-s}
Let $X$ be a reduced projective scheme over $k$ of dimension $d$.
Let $\CH^{LW}_0(X)_{\deg 0}$ denote the kernel of the degree map
${\deg}: \CH^{LW}_0(X) \surj H^0(X, \Z)$.
An Albanese variety $A^d(X)$ of $X$ was constructed in \cite{ESV}. This 
variety is a connected commutative algebraic group over $k$ with an
Abel-Jacobi map $\rho_X: \CH^{LW}_0(X)_{\deg 0} \surj A^d(X)$. Moreover,
$A^d(X)$ is the universal object in the category of commutative algebraic
groups $G$ over $k$ with a rational map $f: X \dashrightarrow G$
which induces a group homomorphism $f_*:\CH^{LW}_0(X)_{\deg 0} \to G$.
Such a group homomorphism is called a regular homomorphism.
For this reason, $A^d(X)$ is also called the {\sl universal regular quotient}
of $\CH^{LW}_0(X)_{\deg 0}$. When $X$ is smooth, $A^d(X)$ coincides with the
classical Albanese variety $\Alb(X)$. 
One says that $\CH^{LW}_0(X)_{\deg 0}$ is {\sl finite-dimensional}, if 
$\rho_{X}$ is an isomorphism. 

\subsection{The universal semi-abelian quotient of 
$\CH^{LW}_0(X)_{\deg 0} $}\label{sec:SAQ}
In positive characteristic, apart from the theorem about its existence,
not much is known about $A^d(X)$ and almost everything
that one would like to know about it is an open question. 
However, the universal semi-abelian quotient of the algebraic group
$A^d(X)$ can be described more explicitly and this is all one needs
to know in order to understand the prime-to-$p$ torsion in $A^d(X)$.
The following description of this quotient 
is recalled from \cite[\S~2]{Mallick}. 

Let $\pi: X^N \to X$ denote the normalization map. Let ${\rm Cl}(X^N)$
and $\Pic_W(X^N)$ denote the class group and the Weil Picard group of
$X^N$. Recall from \cite{Weil} that $\Pic_W(X^N)$ is the subgroup of
${\rm Cl}(X^N)$ consisting of Weil divisors which are algebraically
equivalent to zero in the sense of \cite[Chap.~19]{Fulton}.
Let ${\rm Div}(X)$ denote the free abelian group of Weil divisors on $X$.
Let $\Lambda_{1}(X)$ denote the subgroup of ${\rm Div}(X^N)$ generated
by the Weil divisors which are supported on $\pi^{-1}(X_{\rm sing})$.
This gives us a map $\iota_X: \Lambda_{1}(X) \to 
\frac{{\rm Cl}(X^N)}{\Pic_W(X^N)}$.

Let $\Lambda(X)$ denote the kernel of the canonical map
\begin{equation}\label{eqn:Weil-Pic}
\Lambda_{1}(X) \xrightarrow{(\iota_X, \pi_*)} \frac{{\rm Cl}(X^N)}{\Pic_W(X^N)}
\oplus {\rm Div}(X).
\end{equation}

It follows from \cite[\S~4]{Mallick} that
the universal semi-abelian quotient of $A^d(X)$ is the Cartier dual of
the 1-motive $[\Lambda(X) \to \Pic_W(X^N)]$ and is denoted by $J^d(X)$.
The composite homomorphism
$\rho^{\rm semi}_X: \CH_0^{LW}(X)_{\deg 0} \surj A^d(X) \surj J^d(X)$ is 
universal among regular homomorphisms from $ \CH_0^{LW}(X)_{\deg 0}$
to semi-abelian varieties. $J^d(X)$ is called the {\sl semi-abelian Albanese
variety} of $X$. By \cite[Proposition~9.7]{BK}, there is a factorization
of $\rho^{\rm semi}_X$:
\begin{equation}\label{eqn:alb-lci}
\CH_0^{LW}(X)_{\deg 0} \surj \CH_0(X)_{\deg 0} 
\xrightarrow{{\wt{\rho}}^{\rm semi}_X} J^d(X)
\end{equation}
and $J^d(X)$ is in fact a universal regular semi-abelian variety quotient of 
$\CH_0(X)_{\deg 0}$.

There is a short exact sequence of commutative algebraic 
groups
\begin{equation}\label{eqn:alb-semi}
0 \to A^d_{\rm unip}(X) \to A^d(X) \to J^d(X) \to 0,
\end{equation}
where $A^d_{\rm unip}(X)$ is the unipotent radical of $A^d(X)$.
Since $k$ is algebraically closed, $A^d_{\rm unip}(X)$ is uniquely $n$-divisible
for any integer $n \in k^{\times}$. In particular,  
\begin{equation}\label{eqn:Div-0-3}
{_{n}A^d(X)} \xrightarrow{\simeq} {_{n}J^d(X)}.
\end{equation}

\subsection{A Lefschetz hyperplane theorem}\label{sec:Lef}
Recall that for a smooth projective scheme $X \inj \P^N_k$ of dimension 
$d \ge 3$, a general hypersurface section $Y \subset X$ has the
property that the canonical map $\Alb(Y) \to \Alb(X)$ is an isomorphism.
We now wish to prove a similar result for the semi-abelian Albanese
variety $J^d(S_X)$ of the double. 

Let $X$ be a connected smooth projective scheme of dimension $d \ge 3$ over 
$k$ and let $D \subset X$ be an effective Cartier divisor.
Let $\{E_1, \cdots , E_r\}$ be the set of irreducible components of $D_{\rm red}$
so that each $E_i$ is integral. For each $1 \le j \le r$, let 
$P_j \in E_j \setminus (\cup_{j' \neq j} E_{j'})$ be a chosen closed point,
and let $S = \{P_1, \cdots , P_r\}$.
We can apply \cite[Theorems~1, 7]{KL}
to find a closed embedding $X \inj \P^N_k$ (for $N \gg 0$)
such that for a general set of 
distinct hypersurfaces $H_1, \cdots , H_{d-2} \inj \P^N_k$, we have the
following.
\begin{equation}\label{eqn:Bertini}
\end{equation}
\begin{enumerate}
\item
The successive intersection $X \cap H_1 \cap \cdots \cap H_i$
($1 \le i \le d-2$) is integral and smooth over $k$.
\item
No component of $X \cap  H_1 \cap \cdots \cap H_{d-2}$ is contained in $D$.
\item 
$S \subset X \cap  H_1 \cap \cdots \cap H_{d-2}$.
\item
For each $1 \le j \le r$, the successive intersection 
$E_j \cap H_1 \cap \cdots \cap H_i$ ($1 \le i \le d-2$) is integral.
\end{enumerate}

Setting $X_i = X \cap H_1 \cap \cdots \cap H_i$ and $D_i = X_i \cap D$,
our choice of the hypersurfaces implies that $D_i \subset X_i$ is an
effective Cartier divisor. We let $S_i = X_i \amalg_{D_i} X_i$
for $1 \le i \le d-2$. It follows from \cite[Proposition~2.4]{BK} that
each inclusion $\tau_i:S_i \inj S_{i-1}$ is a local complete intersection
(l.c.i.) and $S_i = S_{i-1} \times_{X_{i-1}} X_i$ for $1 \le i \le d-2$,
where we let $S_0 = S_X$.
Since $S_i \inj S_{i-1}$ is an l.c.i. inclusion with $(S_i)_{\rm sing}
= (D_i)_{\rm red} = (S_i \cap D)_{\rm red} = S_i \cap (S_{i-1})_{\rm sing}$, 
it follows that there are compatible 
push-forward maps $\tau^{LW}_{i, *}: \CH^{LW}_0(S_i) \to \CH^{LW}_0(S_{i-1})$
(see \cite[Lemma~1.8]{ESV}) and $\tau_{i, *}: \CH_0(S_i) \to \CH_0(S_{i-1})$
(see \cite[Proposition~3.17]{BK}).
With the above set-up, we have the following.

\begin{prop}\label{prop:PF-alb-sing}
For $1 \le i \le d-2$, there is an isomorphism of algebraic groups 
$\phi_i: J^{d-i}(S_i) \to J^{d-i+1}(S_{i-1})$ and a commutative diagram
\begin{equation}\label{eqn:PF-alb-sing-0} 
\xymatrix@C1pc{
\CH_0(S_i)_{\deg 0} \ar[r]^-{\tau_{i,*}} \ar[d]_{\rho^{\rm semi}_i} &
\CH_0(S_{i-1})_{\deg 0} \ar[d]^{\rho^{\rm semi}_{i-1}} \\
J^{d-i}(S_i) \ar[r]_-{\phi_i} & J^{d-i+1}(S_{i-1}),}
\end{equation}
whose restriction to the $l$-primary torsion subgroups induces isomorphism
of all arrows, for every prime $l \neq p$.
\end{prop}
\begin{proof}
Before we begin the proof, we note that when $S_X \cap H \inj S_X$ is a 
general hypersurface section in a projective embedding of $S_X$, 
then an analogue of the proposition was proven by 
Mallick in \cite[\S~5]{Mallick}. But we can not apply his result directly
because $S_Y \inj S_X$ is not a hypersurface section. Nevertheless,
we shall closely follow Mallick's construction in the proof of the
proposition. Also, we shall prove a stronger assertion than in 
\cite[Theorem~14]{Mallick} because of the special nature of the double. 

We have seen above that the inclusion $S_i \inj S_{i-1}$ is l.c.i.
which preserves the singular loci, and hence there is a
push-forward $\tau_{i, *}: \CH_0(S_i) \to \CH_0(S_{i-1})$.
We now construct the map $\phi_i$. We construct this map for $i=1$
as the method in the general case is completely identical.
With this reduction in mind, we shall let $Y = X_1 = X \cap H_1, \ 
F = D_1 = D \cap H_1$ so that $S_1 = S_Y = S(Y,F)$.

Since $Y \inj X$ is a general hypersurface section with
$S^N_Y = Y \amalg Y$ and $S^N_X = X \amalg X$, one knows that the map 
$\Pic_W(S^N_Y)^{\vee} = \Alb(Y) \times \Alb(Y) \to \Alb(X) \times 
\Alb(X) = \Pic_W(S^N_X)^{\vee}$ is an isomorphism.  
We thus have to construct a pull-back map $\Lambda(S_X) \to \Lambda(S_Y)$
which is an isomorphism.  

We consider the commutative square
\begin{equation}\label{eqn:PF-alb-sing-1} 
\xymatrix@C1pc{
Y \amalg Y \ar[r]^-{\psi} \ar[d]_{\pi_1} & X \amalg X \ar[d]^{\pi} \\
S_Y \ar[r]_{\tau_1} & S_X,}
\end{equation}
where the vertical arrows are the normalization maps
and $\psi$ is the disjoint sum of the inclusion of the hypersurface section 
$Y = X \cap H_1$.
We claim that this square is Cartesian. To see this, note that the vertical
arrows are disjoint sums of two maps. Hence, it is enough to show that 
~\eqref{eqn:PF-alb-sing-1} is Cartesian when we replace $X \amalg X$ by
$X_+$ and $Y \amalg Y$ by $Y_{+}$. In this case, the vertical arrows
are closed immersions and we know directly from the construction
that $S_Y \times_{S_X} X_+ \simeq (Y \times_X S_X) \times_{S_X} X_+ 
\simeq Y_+$.

If we let $F_i := E_i \cap H_1$ for $1 \le j \le r$, it follows from 
the conditions (2) and (3) of ~\eqref{eqn:Bertini} that each $F_i$ is integral
and $F_i \neq F_j$ if $E_i \neq E_j$.
Now, the refined Gysin homomorphism 
$\psi^{!}: \Lambda_{1}(S_X) = \CH_{d-1}(D \amalg D) \to \CH_{d-2}(F \amalg F)
= \Lambda_1(S_Y)$ (in the sense of \cite{Fulton})
takes $E_i$ to $F_i$ for $i = 1, \cdots , r$ in each copy of $D$.
This map is bijective by our choice of the hypersurface sections 
$Y$ and $F_i$'s.

Now, for any $1 \le j \le r$, we have $\psi^{!}([E_j]) =
\alpha_j[F_j]$ in $\Lambda_1(S_Y) =  \CH_{d-2}(F \amalg F)$,
where $\alpha_j$ is the intersection multiplicity of
$E_j \times_{\P^N_k} H_1$ in $D \cdot H_1$.
Since the composite morphism $E_j \inj X \inj \P^N_k$ is a closed immersion
(and hence separable), this intersection multiplicity must be one so that
we have $\alpha_j =1$ for each $j$.
It follows that $\psi^{!}:\Lambda_{1}(S_X) \to \Lambda_1(S_Y)$ is an 
isomorphism of abelian groups.

We next show that $\psi^{!}$ takes $\Lambda(S_X)$ onto $\Lambda(S_Y)$.
This follows because for any $\Delta \in \Lambda(S_X)$, we have
$\pi_{1, *} \circ \psi^{!}(\Delta) = \tau^{!}_1 \circ \pi_*(\Delta)$ by
\cite[Theorem~6.2]{Fulton} and $\pi_*(\Delta) = 0$ in 
$\CH_{d-1}((S_X)_{\rm sing}) = Z_{d-1}((S_X)_{\rm sing})$. 
This means $\psi^{!}(\Delta) \in \Lambda(S_Y)$.
If we let $\Delta = \sum_l \alpha_l[F_l] \in \Lambda(S_Y)$, then
letting $\Delta ' =  \sum_l \alpha_l[E_l] \in \Lambda_1(S_X)$, we get
$\psi^{!}(\Delta ') = \Delta$. We also have
$\tau^{!}_1 \circ \pi_*(\Delta') = \pi_{1, *} \circ \psi^{!}(\Delta ') =
\pi_{1, *}(\Delta) = 0$. Since
$\tau^{!}_1: \CH_{d-1}((S_X)_{\rm sing}) = \CH_{d-1}(D) \to
\CH_{d-2}(F) = \CH_{d-2}((S_Y)_{\rm sing})$ is bijective,
it follows that $\pi_*(\Delta ') = 0$. It follows that
$\psi^{!}:\Lambda(S_X) \to \Lambda(S_Y)$ is surjective and hence an
isomorphism. We have thus constructed an isomorphism
$\phi_1: J^{d-1}(S_Y) \xrightarrow{\simeq} J^{d}(S_X)$. 

To show the commutative diagram~\eqref{eqn:PF-alb-sing-0},  we can
again assume $i = 1$. We can also replace $\CH_0(S_X)$ and $\CH_0(S_Y)$
by $\sZ_0(S_X)$ and $\sZ_0(S_Y)$, respectively. 

We now need to observe that $J^d(S_{X})$ is a quotient of the Cartier dual 
$J^d_{\rm Serre}(S_{X})$ of the 1-motive
$[\Lambda_1(S_X) \to \Pic_W(X \amalg X)]$ and this dual semi-abelian variety
is the universal object in the category of morphisms from 
$(S_X)_{\rm reg} = (X \setminus D) \amalg (X \setminus D)$ to
semi-abelian varieties (see \cite{Serre-1}).
Since $(S_Y)_{\rm reg} = (Y \setminus D) \amalg (Y \setminus D)
= \tau_1^{-1}((S_X)_{\rm reg})$, it follows from this universality
of $J^d_{\rm Serre}(S_X)$ that there is a commutative diagram
as in ~\eqref{eqn:PF-alb-sing-0} if we replace $J^d(S_X)$ and
$J^{d-1}(S_Y)$ by $J^d_{\rm Serre}(S_X)$ and $J^d_{\rm Serre}(S_Y)$,
respectively. The commutative diagram
\[
\xymatrix@C1pc{
J^d_{\rm Serre}(S_Y) \ar[r]^{\phi_{1,*}} \ar@{->>}[d] & J^d_{\rm Serre}(S_X) 
\ar@{->>}[d] \\
J^d(S_Y) \ar[r]_{\phi_{1,*}} & J^d(S_X)}
\]
now finishes the proof of the commutativity of ~\eqref{eqn:PF-alb-sing-0}.
The final assertion about the isomorphism between the $l$-primary torsion 
subgroups follows from \cite[Theorem~9.8]{BK} and \cite[Theorem~16]{Mallick}.
\end{proof}

\subsection{0-cycles and Albanese variety with modulus }
\label{def:DefChowMod1}
Given an integral normal curve ${C}$ over $k$ and an effective divisor
$E \subset {C}$, we say that a rational function $f$ on ${C}$ has modulus
$E$ if $f \in {\rm Ker}(\sO^{\times}_{C, E} \to \sO^{\times}_E)$.
Here, $\sO_{C, E}$ is the semilocal ring of $C$ at the union of
$E$ and the generic point of $C$.
In particular, ${\rm Ker}(\sO^{\times}_{C, E} \to \sO^{\times}_E)$ is just
$k(C)^{\times}$ if $|E| = \emptyset$.
Let $G({C}, E)$ denote the group of such rational functions.

Let $X$ be a smooth quasi-projective scheme of dimension $d \ge 1$ over $k$
and let $D \subset X$ be an effective Cartier divisor.
Let $\sZ_0(X|D)$ be the free abelian group on 
the set of closed points of $X\setminus D$. Let ${C}$ be an integral normal 
curve over $k$ and
let $\varphi_{{C}}\colon{C}\to {X}$ be a finite morphism such that 
$\varphi_{{C}}({C})\not \subset D$.  
The push-forward of cycles along $\varphi_{{C}}$  
gives a well defined group homomorphism
$\tau_{{C}}\colon G({C},\varphi_{{C}}^*(D)) \to Z_0(X|D)$.

\begin{defn}\label{def:DefChowMod-Definition}
We define the Chow group $\CH_0({X}|D)$ of 0-cycles of ${X}$ with 
modulus $D$ to be the cokernel of the homomorphism 
$\divf \colon\bigoplus_{\varphi_{{C}}\colon {C}\to X}G({C},
\varphi_{{C}}^*(D)) \to Z_0(X|D)$,
where the sum is taken over the set of finite morphisms 
$\varphi_{{C}}\colon {C} \to {X}$ from an integral normal curve such that
$\varphi_{{C}}({C}) \not\subset D$.
\end{defn}

It is clear from the moving lemma of Bloch that there is a canonical 
surjection $\CH_0(X|D) \surj \CH_0(X)$.
If $X$ is projective, we let $\CH_0(X|D)_{\rm deg \ 0}$ denote the kernel of the
composite map $\CH_0(X|D) \surj \CH_0(X) \xrightarrow{\deg} H^0(X, \Z)$. 

If $D$ is reduced, it follows using \thmref{cor:LW-lci-join} and
the constructions of \cite[\S~11]{BK} that there exists an Albanese variety 
with modulus $A^d(X|D)$ together with a surjective Abel-Jacobi map 
$\rho_{X|D}: \CH_0(X|D)_{\rm deg \ 0} \to A^d(X|D)$ which is a
universal regular homomorphism in a suitable sense.
One says that $\CH_0(X|D)_{\deg 0}$ is {\sl finite-dimensional}, if 
$\rho_{X|D}$ is an isomorphism.

\section{Torsion in  the Chow group of a singular 
surface}\label{sec:K-2-surf}
Let $X$ be a reduced projective surface over $k$.
When $X$ is smooth, Suslin \cite[p.~19]{Suslin} showed that 
there is a short exact sequence
\begin{equation}\label{eqn:K-2-ex-sm}
0 \to H^1(X, \sK_2) \otimes_{\Z} {\Z}/n \to H^3_{\etl}(X, \mu_n(2)) \to
{_{n}{\CH_0(X)}} \to 0
\end{equation}
for any integer $n \in k^{\times}$. 
When $X$ has only isolated singularities, this exact sequence
was conjectured by Pedrini and Weibel \cite[(0.4)]{PW} and proven by
Barbieri-Viale, Pedrini and Weibel in \cite[Theorem~7.7]{BPW}.
Expanding on their ideas, our goal in this section is to 
prove such an exact sequence for surfaces with
arbitrary singularities. This exact sequence will be used to prove an
analogue of \thmref{thm:Intro-1} for singular surfaces.

\subsection{Gillet's Chern classes}\label{sec:Gillet}
Let $\sK = \Omega BQP$ denote the simplicial presheaf
on $\Sch_k$ such that $\sK(X) = \Omega BQP(X)$, where $BQP(X)$ is the
Quillen's construction of the simplicial set associated to the exact
category $P(X)$ of locally free sheaves on $X$. For any $i \ge 0$, let
$\sK_i$ denote the Zariski sheaf on $\Sch_k$ associated to the presheaf
$X \mapsto K_i(X) = \pi_i(\sK(X))$. 

We fix integers $d, i \ge 0$ and an integer $n \ge 1$ which is prime to 
the exponential characteristic $p$ of the base field $k$.
Let $\mu_n(d)$ denote the sheaf $\mu^{\otimes d}_n$ on ${\Sch_k}/{\etl}$ such that
$\mu_n(X) = {\rm Ker}(\sO^{\times}(X) \xrightarrow{n} \sO^{\times}(X))$.
Let $\sE_{i}$ denote the simplicial sheaf on 
${\Sch_k}/{\zar}$ associated, by the Dold-Kan correspondence, to the good 
truncation $\tau^{\le 0}{\bf R}w_*\mu_n(i)[2i]$ of the chain complex of Zariski 
sheaves ${\bf R}w_*\mu_n(i)[2i]$. In particular, we have $\pi_q \sE_{i}(X) =
H^{2i-q}_{\etl}(X, \mu_n(i))$ for $X \in \Sch_k$ and $0 \le q \le 2i$.

Let $\sL$ denote the homotopy fiber of the map of simplicial presheaves
$\sK \xrightarrow{n} \sK$
so that $\pi_q \sL(X) = K_{q+1}(X, \Z/n)$ for $q \ge -1$.
Let $\sH^i(\mu_n(d))$ denote the Zariski sheaf
${\bf R}^iw_*\mu_n(d)$ on ${\Sch_k}/{\zar}$.
It follows from Gillet's construction of the Chern classes 
(see \cite[\S~5]{Gillet}) that there is a morphism of the simplicial
presheaves on ${\Sch_k}$:
\begin{equation}\label{eqn:Chern-0}
c_i: \sK \to \sE_{i}
\end{equation}
in the homotopy category of simplicial presheaves.
This map is compatible with the local to global spectral
sequences for the associated Zariski sheaves so that 
we get a map of spectral sequences which at the $E_2$ level is
$H^p(X, \sK_q) \to H^p(X, \sH^{2i-q}(\mu_n(i)))$
and converges to the global Chern class on $\sE_{i}$ given by
$c_{i, X}: K_{q-p}(X)\to H^{2i-q+p}_{\etl}(X,\mu_n(i))$.
These Chern classes in fact factor through the maps
$\sK \to \Sigma \sL \to \sE_i$ (see \cite[\S~2]{PW}) so that there are
Chern class maps $c_i: \sL \to \Omega \sE_i$ in the homotopy category of
simplicial presheaves \cite{BG}.

Let $\wt{\sK}$ denote the simplicial presheaf on $\Sch_k$ such that
$\wt{\sK}(X)$ is the universal covering space of $\sK(X)$ and let 
$\wt{\sL}$ denote the homotopy fiber of the map
$\wt{\sK} \xrightarrow{n} \wt{\sK}$. This yields
$\pi_q(\wt{\sK}(X)) = K_q(X)$ for $q \ge 2$ and $\pi_1(\wt{\sK}(X)) = 0$.
Furthermore, we have
\begin{equation}\label{eqn:Chern-1}
\pi_q(\wt{\sL}(X))  = \left\{
\begin{array}{ll}
K_{q+1}(X, \Z/n) & \mbox{if $q \ge 2$} \\
K_2(X)/n & \mbox{if $q =1$} \\ 
0 & \mbox{otherwise.}
\end{array}
\right.
\end{equation}

Let $\wt{\sK}^{(2)}$ denote the second layer of the Postnikov tower 
$(\cdots \to \wt{\sK}^{(n)} \to \wt{\sK}^{(n-1)} \to \cdots \to \wt{\sK}^{(2)} \to
\wt{\sK}^{(1)} \to \wt{\sK}^{(0)}  = \star)$ of $\wt{\sK}$. There are 
compatible family of maps $f_n: \wt{\sK} \to \wt{\sK}^{(n)}$ inducing
$\pi_q \wt{\sK} \xrightarrow{\simeq} \pi_q \wt{\sK}^{(n)}$ for
$q \le n$ and $\pi_q \wt{\sK}^{(n)} = 0$ for $q > n$.
Let $\wt{\sL}^{(2)}$ denote the homotopy fiber of the map
$\wt{\sK}^{(2)} \xrightarrow{n} \wt{\sK}^{(2)}$ so that there is a
commutative diagram of simplicial presheaves

\begin{equation}\label{eqn:Chern-2}
\xymatrix@C1pc{ 
{\wt{\sL}}^{(2)} \ar[r] & {\wt {\sK}}^{(2)} \ar[r]^-{n} & {\wt {\sK}}^{(2)} \\
{\wt{\sL}} \ar[r] \ar[d] \ar[u] & {\wt{\sK}} \ar[r]^-{n} \ar[d] 
\ar[u] & {\wt {\sK}} \ar[u] \ar[d] \\
\sL \ar[r] & {\sK} \ar[r]^-{n} & \sK.}
\end{equation}

\begin{lem}\label{lem:EM-K-2}
Let $\sK^{\bullet}_2$ denote the complex of presheaves $(\sK_2 \xrightarrow{n}
\sK_2)$ on ${\Sch_k}/{\zar}$. Then ${\wt{\sL}}^{(2)}$ is weak equivalent to
the simplicial presheaf obtained by applying the Dold-Kan correspondence
to the chain complex $\sK^{\bullet}_2[2]$.
\end{lem}
\begin{proof}
Using ~\eqref{eqn:Chern-2}, it suffices to show that
${\wt {\sK}}^{(2)}$ is an Eilenberg-Mac Lane complex of the type
$(\sK_2, 2)$. Let $F^2$ denote the homotopy fiber of the Kan fibration 
${\wt{\sK}}^{(2)} \to {\wt{\sK}}^{(1)}$ in the Postnikov tower. 
We have ${\pi}_{i} {\wt{\sK}}^{(1)} = 0$ for $i \ge 2$ and
${\pi}_{1} {\wt{\sK}}^{(1)} = {\pi}_{1} {\wt{\sK}} = 0$ by 
\cite[Theorem~8.4]{May}.
The long exact homotopy sequence now implies that 
${\pi}_{i} {F^2} \xrightarrow{\simeq}  {\pi}_{i} {\wt{\sK}}^{(2)}$ for all 
$i$ and hence 
${F^2} \to {\wt{\sK}}^{(2)}$ is a weak equivalence by the Whitehead theorem. 
Since $F^2$ is an Eilenberg-Mac Lane complex of type 
$({{\sK}_2}, 2)$ by \cite[Corollary~8.7]{May}, we conclude. 
\end{proof}

\subsection{Cohomology of $\sK_2$ on a surface}\label{sec:Surf*}
Let us now assume that $X$ is a reduced quasi-projective surface over $k$.
Let $\sK_i(\Z/n)$ denote the Zariski sheaf on $\Sch_k$ associated to the
presheaf $X \mapsto K_i(X, \Z/n)$. 
Applying the Brown-Gersten spectral sequence (see \cite[Theorem~3]{BG})
\[
H^p(X, {\pi}_{-q}{\sF}) \Rightarrow H^{p + q}(X, {\sF}),
\]
~\eqref{eqn:Chern-0} and ~\eqref{eqn:Chern-1}, we obtain a commutative
diagram of exact sequences
\begin{equation}\label{eqn:Chern-3}
\xymatrix@C1pc{
H^0(X, {{\sK}_{2}}/{n}) \ar[r]^{d} \ar[d]_{c_2}^{\simeq} & 
H^2(X, {{\sK}_{3}}({\Z}/{n})) 
\ar[r] \ar[d]^{c_2} & H^0(X, {\wt{\sL}}) \ar[r] \ar[d]^{c_2} &
H^1(X, {{\sK}_{2}}/{n}) \ar[r] \ar[d]^{c_2}_{\simeq} &  0 \\
H^0(X, {\sH}^2({\mu}_n(2))) \ar[r]_{d} &
H^2(X, {\sH}^1({\mu}_{n}(2))) \ar[r] & H^3_{\etl} (X, {\mu}_{n}(2)) 
\ar[r] & H^1(X, {\sH}^2({\mu}_{n}(2))) \ar[r] & 0,}
\end{equation}
where the vertical arrows are induced by the Chern class map $c_{2,X}$.
The bottom row is exact because ${\sH}^3({\mu}_{n}(2)) = 0$ by
\cite[Theorem~VI.7.2]{Milne}.
The map $H^0(X, {\wt{\sL}}) \xrightarrow{c_2}  H^3_{\etl} (X, {\mu}_{n}(2))$
is the one induced by the maps
$H^0(X, {\wt{\sL}}) \xrightarrow{c_2} H^0(X, \Omega \sE_2) \simeq
H^{-1}(X, \sE_2) = \pi_1\sE_2(X) = H^{3}_{\etl}(X, \mu_n(2))$. 
The leftmost and the rightmost vertical arrows are isomorphisms by
Hoobler's theorem \cite[Theorem~3]{Hoobler}.

\enlargethispage{25pt}

\begin{lem}\label{lem:K-2-Surface}
There is a functorial map $H^0(X, {\wt{\sL}}^{(2)}) 
\to \H^2(X, \sK^{\bullet}_2)$ and a commutative
diagram with exact rows
\begin{equation}\label{eqn:K-2-Surface-0}
\xymatrix@C1pc{
H^0(X, \pi_1 \wt{\sL}^{(2)}) \ar[r]^{d} \ar[d] &
H^2(X, \pi_2 \wt{\sL}^{(2)}) \ar[r] \ar[d] &
H^0(X, {\wt{\sL}}^{(2)}) \ar[r] \ar[d] & 
H^1(X, \pi_1 \wt{\sL}^{(2)}) \ar[r] \ar[d] & 0 \\
H^0(X, \sK_2/n) \ar[r]_{d} & H^2(X, {_n\sK_2}) \ar[r] &
\H^2(X, \sK^{\bullet}_2) \ar[r] & H^1(X, \sK_2/n) \ar[r] & 0}
\end{equation}
such that  all vertical arrows are isomorphisms.
\end{lem}
\begin{proof}
The bottom exact sequence follows from the exact triangle
\begin{equation}\label{eqn:K-2-Surface-1}
{_{n} {\sK}_2} \to {\sK}^{\bullet}_2 \to {{\sK}_2}/{n}[-1] \to 
{_{n} {\sK}_2}[1]
\end{equation}
in the derived category of Zariski sheaves of abelian groups on 
${\Sch_k}/{\zar}$.
On the other hand, Lemma~\ref{lem:EM-K-2} says that 
${\wt {\sK}}^{(2)}$ is an Eilenberg-Mac Lane complex of the type
$(\sK_2, 2)$ and there is a homotopy equivalence
${\wt{\sL}}^{(2)} \to \sK^{\bullet}_2[2]$. This in particular implies
that $\pi_1 \wt{\sL}^{(2)} \simeq \sK_2/n$, $\pi_2 \wt{\sL}^{(2)} \simeq
{_n\sK_2}$ and $\pi_i \wt{\sL}^{(2)} = 0$ for $i \neq 1,2$.
Applying the Brown-Gersten spectral sequence to $H^0(X, {\wt{\sL}}^{(2)})$
and using  Lemma~\ref{lem:EM-K-2}, we conclude the proof.
The commutativity follows because both rows are obtained by applying the
hypercohomology spectral sequence to the map 
${\wt{\sL}}^{(2)} \to \sK^{\bullet}_2[2]$.
\end{proof}

The key step in extending ~\eqref{eqn:K-2-ex-sm} to arbitrary surfaces is the
following result.

\begin{lem}\label{lem:Surface-0*1}
The map $H^2(X, {\sK}_3 ({\Z}/{n})) \to H^2(X, {_{n}{\sK}_2})$, induced by 
the universal coefficient theorem, has 
a natural factorization
\[
H^2(X, {\sK}_3 ({\Z}/{n})) \xrightarrow{c_2} H^2(X, {\sH}^1({\mu}_{n}(2)))
\xrightarrow{\nu}  H^2(X, {_{n}{\sK}_2})
\]
such that the map $\nu$ is an isomorphism.
\end{lem}
\begin{proof}
By \cite[Lemma~6.2, Variant~6.3]{BPW}, there is a commutative diagram
\begin{equation}\label{eqn:K-2-Surface-2}
\xymatrix@C1pc{
{\sH}^1({\mu}_{n}(2)) \ar[r]^-{\phi} \ar[dr] & {\sK}_3 ({\Z}/{n})
\ar[d]^{\psi} \ar[r]^-{c_2} & {\sH}^1({\mu}_{n}(2)) \ar@{-->}[dl] \\
& {_{n}{\sK}_2}}
\end{equation}
such that $c_2 \circ \phi = -1$ and the composite $\psi \circ \phi$ is given by 
$a \mapsto \{a, \zeta\} \in {_{n}{\sK}_2}$ 
(where $\zeta \in k^{\times}$ is a primitive $n^{\rm th}$ root of unity).
This composite map is surjective and
an isomorphism on the regular locus of $X$ by \cite[Lemma~6.2]{BPW}.
It follows that the induced map 
$\psi \circ \phi: H^2(X, {\sH}^1({\mu}_{n}(2))) \to H^2(X, {_{n}{\sK}_2})$ is 
an isomorphism (see \cite[Lemma~1.3]{PW}). 
We thus have a diagram
\begin{equation}\label{eqn:Surface-0*1-0}
\xymatrix@C1pc{
H^2(X, {\sH}^1({\mu}_{n}(2))) \ar[r]^{\phi} \ar[dr]_{\simeq} &
H^2(X, {\sK}_3 ({\Z}/{n})) \ar@{->>}[r]^-{c_2} \ar[d]^{\psi} &
H^2(X, {\sH}^1({\mu}_{n}(2))) \ar@{-->}[dl]^{\nu} \\
& H^2(X, {_{n}{\sK}_2})}
\end{equation}
in which the triangle on the left is commutative. To prove the lemma,
it is therefore sufficient to show that ${\rm Ker}(c_2) =
{\rm Ker}(\psi)$. Equivalently, ${\rm Ker}(c_2) \subseteq {\rm Ker}(\psi)$. 

We set $\sF = {\rm Ker}(c_2)$ so that we have a split exact
sequence 
\begin{equation}\label{eqn:Surface-0*1-1}
0 \to \sF \to {\sK}_3 ({\Z}/{n}) \xrightarrow{c_2} {\sH}^1({\mu}_{n}(2)) \to 0.
\end{equation}
By \cite[Lemma~1.3]{PW}, it suffices to show that the composite map
$\sF \to  {\sK}_3 ({\Z}/{n}) \xrightarrow{\psi} {_{n}{\sK}_2}$
of Zariski sheaves is zero on the smooth locus of $X$. 
Since this is a local question, it suffices to show that for 
a regular local ring $A$ which is essentially of finite type over $k$,
the map ${\rm Ker}(c_2) \to {_{n}{K}_2(A)}$ is zero.
By the Gersten resolution of $K_2(A)$, Bloch-Ogus resolution for
$H^1_{\etl}(A, \mu_n(2))$ and Gillet's resolution for $K_3(A, {\Z}/n)$
(see \cite{Quillen}, \cite{Bloch-O} and \cite{Gillet-1}), 
we can replace $A$ by its fraction field $F$.

We now have a commutative diagram 
\begin{equation}\label{eqn:Surface-0*1-2}
\xymatrix@C1pc{
K^M_3(F) \ar[d] \ar[r] & K_3(F, {\Z}/n) \ar[r]^-{c_2} \ar@{=}[d] &
H^1_{\etl}(F, \mu_n(2)) \ar@{-->}[d] \ar[r] & 0 \\
K_3(F) \ar[r] &  K_3(F, {\Z}/n) \ar[r]_-{\psi} & {_{n}{K}_2(F)} \ar[r] & 0.}
\end{equation}

One of the main results of \cite{Levine-Ind} (and also \cite{MS}) shows 
that the top row in ~\eqref{eqn:Surface-0*1-2} is exact. Since the bottom row is
clearly exact, we get the desired conclusion.
\end{proof}

Our first main result of this section is the following extension of 
~\eqref{eqn:K-2-ex-sm} to surfaces with arbitrary singularities.

\begin{thm}\label{thm:K-2-coh-0}
Let $X$ be a reduced quasi-projective surface over $k$ and let $n \ge 1$ be an 
integer prime to $p$. Then, there is a short exact sequence
\[
0 \to H^1(X, \sK_2) \otimes_{\Z} {\Z}/n \to H^3_{\etl}(X, \mu_n(2)) \to
{_{n}{\CH_0(X)}} \to 0.
\]
\end{thm}
\begin{proof}
It follows from \lemref{lem:K-2-Surface} and the map between the
Brown-Gersten spectral sequences associated to the morphism of
simplicial presheaves $\wt{\sL} \to \wt{\sL}^{(2)}$
that there is a commutative diagram
\begin{equation}\label{eqn:K-2-Surface-3}
\xymatrix@C1pc{
H^0(X, {{\sK}_{2}}/{n}) \ar[r]^<<<{d} \ar[d]_{\simeq} & 
H^2(X, {{\sK}_{3}}({\Z}/{n})) 
\ar[r] \ar[d]^{\psi} & H^0(X, {\wt{\sL}}) \ar[r] \ar[d] &
H^1(X, {{\sK}_{2}}/{n}) \ar[r] \ar[d]^{\simeq} &  0 \\
H^0(X, \sK_2/n) \ar[r]_{d} & H^2(X, {_n\sK_2}) \ar[r] &
H^0(X, \wt{\sL}^{(2)}) \ar[r] & H^1(X, \sK_2/n) \ar[r] & 0.}
\end{equation}

Combining this with ~\eqref{eqn:Chern-3} and \lemref{lem:K-2-Surface}, 
we obtain a commutative diagram
\begin{equation}\label{eqn:K-2-Surface-4}
\xymatrix@C.8pc{
H^0(X, \sK_2/n) \ar[r]^{d} & H^2(X, {_n\sK_2}) \ar[r] &
\H^2(X, \sK^{\bullet}_2) \ar[r] & H^1(X, \sK_2/n) \ar[r] & 0 \\
H^0(X, {{\sK}_{2}}/{n}) \ar[r]^<<<{d} \ar[d]^{c_2}_{\simeq} \ar[u]^{\simeq} & 
H^2(X, {{\sK}_{3}}({\Z}/{n})) \ar[r] \ar[d]^{c_2} \ar[u]_{\psi} & 
H^0(X, {\wt{\sL}}) \ar[r] \ar[d]^{c_2} \ar[u] &
H^1(X, {{\sK}_{2}}/{n}) \ar[r] \ar[d]^{\simeq}_{c_2} \ar[u]_{\simeq} &  0 \\
H^0(X, {\sH}^2({\mu}_n(2))) \ar[r]_{d} &
H^2(X, {\sH}^1({\mu}_{n}(2))) \ar[r] & H^3_{\etl} (X, {\mu}_{n}(2)) 
\ar[r] & H^1(X, {\sH}^2({\mu}_{n}(2))) \ar[r] & 0.}
\end{equation}

It follows from \lemref{lem:Surface-0*1} that $\psi$ and the corresponding
vertical arrow downward are surjective with identical kernels. A simple
diagram chase shows that the maps $H^0(X, {\wt{\sL}}) \to 
\H^2(X, \sK^{\bullet}_2)$ and $H^0(X, {\wt{\sL}}) \to H^3_{\etl} (X, {\mu}_{n}(2))$
are surjective with identical kernels. In particular, there is a
natural isomorphism $c_{2, X}: \H^2(X, \sK^{\bullet}_2) \xrightarrow{\simeq}
H^3_{\etl} (X, {\mu}_{n}(2))$. 
Using the exact sequence
\[
0 \to H^1(X, \sK_2) \otimes_{\Z} {\Z}/n \to \H^2(X, \sK^{\bullet}_2) \to
{_{n}{H^2(X, \sK_2)}} \to 0
\]
and the isomorphisms $H^2(X, \sK_2) \xrightarrow{\simeq} \CH^{LW}_0(X)
\xrightarrow{\simeq} \CH_0(X)$ (see \cite[Theorem~7]{Levine-1} and  
\cite[Theorem~3.17]{BK}), we now conclude the proof.
\end{proof}

\subsection{Theorem~\ref{thm:Intro-1} for singular surfaces}
\label{sec:Div}
As an application of \thmref{thm:K-2-coh-0}, we now prove a version of 
\thmref{thm:Intro-1} for singular surfaces. This result was
proven for normal projective surfaces in \cite[Theorem~7.9]{BPW}.
Let $l\neq p$ be a prime number. Let $X$ be a reduced projective surface over
$k$. We shall make no distinction between $\CH^{LW}_0(X)$ and $\CH_0(X)$
in view of \cite[Theorem~3.17]{BK}.

\begin{lem}\label{lem:Div-0}
$H^3_{\etl}(X, {\Q_l}/{\Z_l}(2))$ is divisible by $l^n$ for every $n \ge 1$.
\end{lem}
\begin{proof}
The exact sequence of \thmref{thm:K-2-coh-0} is compatible with
the maps ${\Z}/{l^n} \to {\Z}/{l^{n+1}}$. Taking the direct limit,
we obtain a short sequence
\begin{equation}\label{eqn:lem:Div-0-1}
0 \to H^1(X, \sK_2) \otimes_{\Z} {\Q_l}/{\Z_l} \to 
H^3_{\etl}(X, {\Q_l}/{\Z_l}(2)) \xrightarrow{\tau_X} \CH_0(X)\{l\} \to 0.
\end{equation}
The group on the left is divisible. It is known that $\CH_0(X)_{\deg 0}$ is 
generated by the images of the maps $\Pic^0(C) \to \CH_0(X)$, where
$C \subsetneq X$ is a reduced Cartier curve. Since $\Pic^0(C)$ is
$l^n$-divisible, it follows that
$\CH_0(X)_{\deg 0}$ is $l^n$-divisible.
In particular, $\CH_0(X)\{l\} = \CH_0(X)_{\deg 0}\{l\}$ is also $l^n$-divisible.
The lemma follows.
\end{proof} 


\begin{thm}\label{thm:Cyc-iso}
Given a reduced projective surface $X$ over $k$ and a prime $l \neq p$, the
following hold.
\begin{enumerate}
\item 
$H^1(X, \sK_2) \otimes_{\Z} {\Q_l}/{\Z_l} = 0$.
\item
The map $\tau_X: H^3_{\etl}(X, {\Q_l}/{\Z_l}(2)) \to \CH_0(X)\{l\}$
is an isomorphism.
\end{enumerate}
\end{thm}
\begin{proof}
In view of ~\eqref{eqn:lem:Div-0-1}, the theorem is equivalent to showing that
$\tau_X$ is injective. To show this, it suffices to
prove the stronger assertion that the map 
$\delta:= \rho^{\rm semi}_X \circ \tau_X: H^3_{\etl}(X, {\Q_l}/{\Z_l}(2)) \to 
\CH_0(X)\{l\} \to J^2(X)\{l\}$ is an isomorphism. 

In order to prove this, we first prove a stronger
version of its surjectivity assertion, namely, that for every $n \ge 1$,
the map $H^3_{\etl}(X, \mu_{l^n}(2)) \to {_{l^n}{J^2(X)}}$ is surjective.
In view of \thmref{thm:K-2-coh-0}, it suffices to show that the map
${_{l^n}{\CH_0(X)}} \to {_{l^n}{J^2(X)}}$ is surjective.
To prove this, we use \cite[Theorem~14]{Mallick}, which says that we
can find a reduced Cartier curve $C \subsetneq X$ (a suitable hypersurface
section in a projective embedding) such that
$C \cap X_{\rm reg} \subseteq C_{\rm reg}$ and the induced map
${_{l^n}{\Pic^0(C)}} \to {_{l^n}{J^2(X)}}$ is surjective.
The commutative diagram
\[
\xymatrix@C1pc{
{_{l^n}{\Pic^0(C)}} \ar[r] \ar[d]_{\simeq} & {_{l^n}{\CH_0(X)}} \ar[d] \\
{_{l^n}{J^1(C)}}  \ar[r] & {_{l^n}{J^2(X)}}}
\]
then proves the desired surjectivity where the left vertical arrow is
an isomorphism because the ${\rm Ker}(\Pic^0(C) \surj J^1(C))$
is unipotent, and hence uniquely $l^n$-divisible.

In the rest of the proof, we shall ignore the Tate twist in {\'e}tale
cohomology. Let $\delta_n$ denote the composite map
$H^3_{\etl}(X, \mu_{l^n}(2)) \to {_{l^n}{\CH_0(X)}} \to {_{l^n}{J^2(X)}}$
so that $\delta = {\underset{n}\varinjlim} \ \delta_n$.
Using the above surjectivity, we get a direct system of short exact
sequences
\begin{equation}\label{eqn:Div-0-6}
0 \to T_{l^n}(X) \to H^3_{\etl}(X, \mu_{l^n}(2)) \to {_{l^n}{J^2(X)}} \to 0
\end{equation}
whose direct limit is the short exact sequence
\begin{equation}\label{eqn:Div-0-5}
0 \to T_{l^\infty}(X) \to H^3_{\etl}(X, {\Q_l}/{\Z_l}(2)) \to J^2(X)\{l\} \to 0.
\end{equation}

To prove the theorem, we are only left with showing that $T_{l^\infty}(X) = 0$.
Since $H^3_{\etl}(X, {\Q_l}/{\Z_l}(2)) \simeq {\underset{n}\varinjlim} \
H^3_{\etl}(X, \mu_{l^n}(2)) \simeq {\underset{n}\varinjlim} \
H^3_{\etl}(X, {\Z}/{l^n})$, it follows that
$H^3_{\etl}(X, {\Q_l}/{\Z_l}(2))$ \\
$= H^3_{\etl}(X, {\Q_l}/{\Z_l}(2))\{l\}$.
In particular, $T_{l^\infty}(X)$ is an $l$-primary torsion group.

We next consider a commutative diagram of short exact sequences
\[
\xymatrix@C.8pc{
0 \ar[r] & T_{l^\infty}(X) \ar[r] \ar[d]_{l^n_1} & 
H^3_{\etl}(X, {\Q_l}/{\Z_l}(2)) \ar[r]^-{\delta} \ar[d]^{l^n_2} &
J^2(X)\{l\} \ar[d]^{l^n_3} \ar[r] & 0 \\
0 \ar[r] & T_{l^\infty}(X) \ar[r] & 
H^3_{\etl}(X, {\Q_l}/{\Z_l}(2)) \ar[r]^-{\delta} &
J^2(X)\{l\} \ar[r] & 0,}
\]
in which the vertical arrows are simply multiplication by $l^n$.
It follows from \lemref{lem:Div-0} that the middle vertical arrow is
surjective. We have shown above that
$H^3_{\etl}(X, \mu_{l^n}(2)) \to {_{l^n}{J^2(X)}}$ is surjective.
In particular, the map ${\ker}(l^n_2) \to {\rm Ker}(l^n_3)$ is surjective.
It follows that $T_{l^\infty} \otimes_{\Z} {\Z}/{l^n} = 0$ for every $n \ge 1$. 
We thus only have to show that $T_{l^\infty}$ is finite to finish the proof.

Now, an easy argument that involves dualizing the argument of 
\cite[Theorem~2.5.4]{BAS} (see \cite[Proof of Theorem~15, Claim~2]{Mallick}),
shows that there is a positive integer $N_X$, depending
only on $X$ and not on the integer $n$, such that 
\begin{equation}\label{eqn:Div-0-7}
|H^3_{\etl}(X, \mu_{l^n}(2)))| \le N_X \cdot |{_{l^n}{J^2(X)}}|.
\end{equation}
Combining this with ~\eqref{eqn:Div-0-6}, it follows that
for every $n \ge 1$, either ${_{l^n}{J^2(X)}} = 0$ and hence $T_{l^n}(X) = 0$ or
${_{l^n}{J^2(X)}} \neq 0$  and
$|T_{l^n}(X)| \cdot |{_{l^n}{J^2(X)}}| \le N_X \cdot |{_{l^n}{J^2(X)}}|$. 
In particular, 
either $T_{l^n}(X) = 0$ or $|T_{l^n}(X)| \le N_X$ for every $n \ge 1$. 
But this implies that $T_{l^\infty}(X)$ is finite.
This completes the proof of the theorem. 
\end{proof}

\subsection{Relation with Bloch's construction}
In \cite[\S~2]{Bloch-tor}, Bloch constructed a map
$\lambda_X: \CH_0(X)\{l\} \to H^{2d-1}_{\etl}(X, {\Q_l}/{\Z_l}(2))$
for a smooth projective scheme $X$ of dimension $d \ge 1$ over $k$. 
We end this section with the following lemma that 
explains the relation between Bloch's construction 
and ours when $X$ is a smooth surface.

\begin{lem}\label{lem:Bloch-surface}
If $X$ is a smooth projective surface over $k$, then $\tau_X:
H^3_{\etl}(X, {\Q_l}/{\Z_l}(2)) \to \CH_0(X)\{l\}$ coincides with the negative
of the inverse of Bloch's map
$\lambda_X:  \CH_0(X)\{l\} \to H^3_{\etl}(X, {\Q_l}/{\Z_l}(2))$.
\end{lem}
\begin{proof}
To prove the lemma, we have to go back to the construction of 
$\tau_X: H^3_{\etl}(X, \mu_n(2)) \to {_{l^n}{\CH_0(X)}}$  
in \S~\ref{sec:Surf*} for $n \in k^{\times}$.
Since $X$ is smooth, there is an isomorphism of Zariski sheaves
${\sO^{\times}_X}/n \xrightarrow{\simeq} {\sH}^1({\mu}_{n}(2))  
\xrightarrow{\simeq} {_n\sK_2}$ by \cite[Lemma~6.2]{BPW}. 
Since $H^2(X, \sO^{\times}_X) = 0$, we see that the top
cohomology of all these sheaves vanish. 

This implies that the map $\H^2(X, \sK^{\bullet}_2) \to 
{_{n}{H^2(X, \sK_2)}}$ is simply the 
composite $\H^2(X, \sK^{\bullet}_2)$ \\
$\xrightarrow{\simeq} 
H^1(X, \sK_2/n) \to H^2(X, {_{n}{\sK_2}}) \to {_{n}{H^2(X, \sK_2)}}$.
Moreover, there is a diagram
\begin{equation}\label{eqn:BS-0}
\xymatrix@C1pc{
\H^2(X, \sK^{\bullet}_2) \ar[r]^-{\simeq} \ar@{-->}[d] & 
H^1(X, \sK_2/n) \ar[d]_{\beta_n}^{\simeq} \\
H^3_{\etl} (X, {\mu}_{n}(2)) \ar[r]_-{\simeq} &
H^1(X, {\sH}^2({\mu}_{n}(2))),}
\end{equation}
where the right vertical arrow is the Galois symbol map.
This induces a unique isomorphism $\alpha_n: H^3_{\etl} (X, {\mu}_{n}(2))
\xrightarrow{\simeq} \H^2(X, \sK^{\bullet}_2)$. 
One checks from ~\eqref{eqn:K-2-Surface-4} that $\tau_X$
is simply the composite 
\begin{equation}\label{eqn:BS-1}
H^3_{\etl} (X, {\mu}_{n}(2))
\xrightarrow{\alpha_n} \H^2(X, \sK^{\bullet}_2) \xrightarrow{\simeq}
H^1(X, \sK_2/n) \xrightarrow{\delta_n} {_{n}{H^2(X, \sK_2)}}.
\end{equation}

We also have a commutative diagram
\begin{equation}\label{eqn:BS-2}
\xymatrix@C.8pc{
{K_2(k(X))}/n \ar[r] \ar[d]_{\beta_n} & {\underset{x \in X^{(1)}} \coprod}\
{k(x)^{\times}}/n \ar[d]^{\beta_n} \ar[r] & 
{\underset{y \in X^{(2)}} \coprod}\ {\Z}/n \ar[d]^{\beta_n} \\
H^2_{\etl}(k(X), {\mu}_{n}(2)) \ar[r] & 
{\underset{x \in X^{(1)}} \coprod}\ H^1_{\etl}(k(x), {\mu}_{n}(1)) 
\ar[r] & 
{\underset{y \in X^{(2)}} \coprod}\ H^0_{\etl}(k(y), {\Z}/n),}
\end{equation}
where all vertical arrows are the Galois symbols and are isomorphisms. 

Since the Gersten resolution is universally exact (see \cite{Quillen} and
\cite{Grayson}), the middle
cohomology of the top row is $H^1(X, \sK_2/n)$ and the Bloch-Ogus
resolution (see \cite{Bloch-O}) 
shows that the middle cohomology of the bottom row is
$H^1(X, {\sH}^2({\mu}_{n}(2)))$. 
The isomorphism $\beta_n$ in ~\eqref{eqn:BS-0} is induced by the
vertical arrows of ~\eqref{eqn:BS-2}.

For a fixed prime $l \neq p$, the map $\delta_{l^m}$ in ~\eqref{eqn:BS-1}
becomes an isomorphism on the limit over $m \ge 1$. Assuming this
isomorphism, it follows from ~\eqref{eqn:BS-1} that
$\tau_X = {\underset{m}\varinjlim} \ \alpha_{l^m}$. 
On the other hand, it follows from
Bloch's construction in \cite[\S~2]{Bloch-tor} that
$\lambda_X = {\underset{m}\varinjlim} \ \beta_{l^m}$.
Since each $\alpha_n$ (with $n \in k^{\times}$) 
is the inverse of $\beta_n$ by definition, the lemma follows. 
\end{proof}

\section{Roitman's torsion for separably weakly normal 
surfaces}
\label{sec:finite-dim}
The Roitman torsion theorem says that for a smooth projective scheme of
dimension $d \ge 1$ over $k$, the Abel-Jacobi map $\rho_X: \CH_0(X)_{\deg 0}
\to J^d(X)$ is an isomorphism on the torsion subgroups.
This theorem was extended to normal projective schemes in \cite{KS}.
However, when $X$ has arbitrary singularities, this theorem is known only for
the torsion prime-to-$p$ (see \cite{BS-1} and \cite{Mallick}).
In this section, we extend the Roitman torsion theorem to separably 
weakly normal (see below for definition) 
surfaces. Later in this text, we shall prove a suitable 
generalization of this theorem in higher dimension.
This generalization will be used to prove \thmref{thm:Intro-2}.

\subsection{Separably weakly normal schemes}\label{sec:WN}
The weak normality is a singularity type of schemes, which in
characteristic $p > 0$, is closely related to various $F$-singularities.
These $F$-singularities are naturally encountered while running the minimal 
model program in positive characteristic.
Most of the $F$-regularity conditions imply weak normality.
In characteristic zero, weak normality coincides with the more familiar notion
of semi-normality. In this text, we shall study the Chow group of 0-cycles
on certain singular schemes whose singularities are closely related to
weak normality in positive characteristic.

Let $A$ be a reduced commutative Noetherian ring.
Let $B$ be the integral closure of $A$ in its
total quotient ring. Recall from \cite{Manaresi} that
the semi-normalization of $A$ is the largest among the subrings 
$A'$ of $B$ containing $A$ such that
\begin{enumerate}
\item 
$\forall \ x \in \Spec(A)$, there exists exactly one $x' \in \Spec(A')$ over 
$x$ and
\item
the canonical homomorphism $k(x) \to k(x')$ is an isomorphism.
\end{enumerate}
The weak normalization of $A$ is the largest among the subrings 
$A'$ of $B$ containing $A$ such that
\begin{enumerate}
\item 
$\forall \ x \in \Spec(A)$, there exists exactly one $x' \in \Spec(A')$ over 
$x$ and
\item
the field extension $k(x) \to k(x')$ is finite purely inseparable.
\end{enumerate}

We let $A_s \subset B$ and $A_w \subset B$ denote the semi-normalization
and weak normalization of $A$, respectively. One says that $A$ is
semi-normal (resp. weakly normal) if $A = A_s$ (resp. $A = A_w$).
It is clear from the above definition that $A \subset A_s \subset A_w \subset
B$. Moreover, $A_s = A_w$ if $A$ is a $\Q$-algebra.
To get a more geometric understanding of these singularities, we make the 
following definition.

\begin{defn}\label{defn:Pins}
Let $k$ be a field and let $R$ be a $k$-algebra which is finite as
a $k$-vector space. We shall say that $R$ is {\sl weakly separable}
over $k$ if it is reduced and either ${\rm char}(k) = 0$, 
or ${\rm char}(k) > 0$ and there is no inclusion of rings 
$k \subsetneq K \subset R$
such that $K$ is a purely inseparable field extension of $k$.   
\end{defn}

Recall that a commutative Noetherian ring $A$ is called an $S_2$ ring
if for every prime ideal $\fp$ of $A$, one has ${\rm depth}(A_{\fp}) \ge
{\rm min}({\rm ht}(\fp), 2)$. It follows easily from this definition that
a Cohen-Macaulay ring is $S_2$.

\begin{prop}\label{prop:weakns}
Let $A$ be reduced commutative Noetherian ring. 
Let $B$ be the integral closure of $A$ in its 
total quotient ring and let
$I \subset B$ be the largest ideal which is contained in $A$. 
Assume that $B$ is a finite $A$-module. Then the following hold.
\begin{enumerate}
\item
If $A$ is semi-normal, then $B/I$ is reduced. If $A$ is $S_2$, the 
converse also holds.
\item
If $A$ is weakly normal, then $B/I$ is reduced and the inclusion map
$A/I \to B/I$ is generically weakly separable. 
If $A$ is $S_2$, the converse also holds.
\end{enumerate}
\end{prop}
\begin{proof}
If $A$ is semi-normal, then $B/I$ is reduced by \cite[Lemma~1.3]{Traverso}.
Suppose now that $A$ is $S_2$ and $B/I$ is reduced. 
By \cite[Theorem~2.6]{GC}, it suffices to show that $A_{\fp}$ is
semi-normal for every height one prime ideal $\fp$ in $A$.

Let $\fp \subset A$ be a prime ideal of height one.
Let $\{x_1, \cdots , x_r\} \subset B$ be a subset whose image in
$B/A$ generates it as an $A$-module.
We now note that $(A_{\fp}: B_{\fp}) = (A:B)_{\fp}$, where $I := (A:B) =
\{a \in A| aB \subset A\} = {\rm ann}(B/A)$. 
The first equality uses the fact that
$B/A$ is a finite $A$-module so that $I = \stackrel{r}{\underset{i = 1}
\cap} (A: x_i)$ and $(J \cap J')_{\fp} = J_{\fp} \cap J'_{\fp}$
(see \cite[Ex.~4.8]{Matsumura}).
Since ${B_{\fp}}/{I_{\fp}} =(B/I)_{\fp}$ is reduced (by our assumption)
and $A_{\fp}$ is 1-dimensional, it follows from \cite[Proposition~7.2]{BM}
and \cite[Theorem~3.6]{Traverso} that $A_{\fp}$ is semi-normal.  
This proves (1).

Suppose now that $A$ is weakly normal. Then it is clearly semi-normal.
In particular, $B/I$ is reduced by (1). We now show that 
$A/I \inj B/I$ is generically weakly separable. Note that as $A$ is reduced and
generically weakly normal, $I$ must have height at least one in $A$.
Let $\fp$ be a minimal prime
of $I$ in $A$ and let $k(\fp)$ be its residue field. Then 
$\fp A_{\fp} = \fp B_{\fp}$ is the Jacobson radical of $B_{\fp}$
and hence ${B_{\fp}}/{\fp B_{\fp}}$ is a product of finite field extensions of 
$k(\fp)$.

We consider the commutative diagram
\begin{equation}\label{eqn:prop:weakns-1}
\xymatrix@C1pc{
A_{\fp} \ar[r] \ar[d] & B_{\fp} \ar[d] & \\
k(\fp) \ar[r] & {B_{\fp}}/{\fp B_{\fp}} \ar[r]^-{\simeq} & 
\stackrel{s}{\underset{i =1}\prod} k(\fq_i),}
\end{equation}
in which the square is Cartesian, the horizontal arrows are injective finite 
morphisms and the vertical arrows are surjective. 
Since $A_{\fp}$ is weakly normal (see \cite[Theorem~IV.3]{Manaresi})
and $B_{\fp}$ is its integral closure, it means that $A_{\fp}$ is weakly normal
in $B_{\fp}$. We conclude from \cite[Proposition~3]{Yanagihara} that
$k(\fp)$ is weakly normal in $\stackrel{s}{\underset{i =1}\prod} k(\fq_i)$.
We now apply \cite[Lemma~2]{Yanagihara} to deduce that the lower
horizontal arrow in ~\eqref{eqn:prop:weakns-1} is weakly separable,
as desired.

Conversely, suppose that $A$ is $S_2$ and the inclusion  
$A/I \inj B/I$ is generically weakly separable map of reduced schemes. 
By \cite[Corollary~IV. 4]{Manaresi}, it suffices 
to show that $A_{\fp}$ is weakly normal for every height one prime ideal 
$\fp$ in $A$. 

Let $\fp \subset A$ be a prime ideal of height one.
Since $I_{\fp} = (A_{\fp}: B_{\fp})$ as shown above, 
it follows that $A_{\fp} = B_{\fp}$ (hence $A_{\fp}$ is weakly normal)
if $I \not\subset \fp$. 
We can therefore assume that $I \subset \fp$. Since $A$ is generically 
normal (and hence weakly normal), $I$ must have height at least one in $A$. 
It follows that $\fp$ must be a minimal prime ideal of $I$.

It follows from (1) that $A_{\fp}$ is semi-normal.
In particular,  $I_{\fp} =
\fp A_{\fp} = \fp B_{\fp}$. Furthermore, $\fp B_{\fp}$ is
the Jacobson radical of $B_{\fp}$ such that ${B_{\fp}}/{\fp B_{\fp}}$
is a finite product of finite field extensions of $k(\fp)$.
This gives rise to a Cartesian square of rings as in ~\eqref{eqn:prop:weakns-1}.
Our assumption says that the lower horizontal arrow in 
~\eqref{eqn:prop:weakns-1} is weakly separable.  We conclude again from 
\cite[Lemma~2, Proposition~3]{Yanagihara} that $A_{\fp}$ is weakly normal.
This proves 
\nolinebreak 
(2).  
\end{proof}


\begin{exm}\label{example:s-normal}
Using \propref{prop:weakns}, we can construct many examples of semi-normal
rings which are not weakly normal.
Let $k$ be perfect field of characteristic $p > 0$ and
consider the Cartesian square of rings 
\begin{equation}\label{eqn:s-normal-0}
\xymatrix@C1pc{
A \ar[r]^-{\psi} \ar[d] & k[x,t] \ar@{->>}[d]  \\
k[x] \ar[r]^-{\phi} & \frac{k[x,t]}{(t^p-x)}}
\end{equation} 
in which the right vertical arrow is the canonical surjection
and $\phi$ is the canonical inclusion.
In particular, one has $A = \{(f(x), g(x,t))| f(t^p) = g(t^p, t)\}$.
Since $\phi$ is a finite purely inseparable map of reduced rings which is
not an isomorphism, it follows that $\psi$ is not an isomorphism.
It follows from \propref{prop:weakns} that
$A$ is semi-normal but not weakly normal.
\end{exm}


\begin{exm}\label{example:w-normal}
We now provide an example of a weakly normal $S_2$ ring $A$ with normalization
$B$ and reduced conductor $I$ such that the map $A/I \to B/I$ is not 
generically separable.
Let $k$ be a perfect field of characteristic $p > 0$
and consider the Cartesian square 
\begin{equation}\label{eqn:w-normal-0}
\xymatrix@C1pc{
A \ar[r]^-{\psi} \ar[d] & k[x,t] \ar[d] \\
k[x] \ar[r]^-{\phi} & \frac{k[x,t]}{t(t^p-x)},}
\end{equation}
where the right vertical arrow is the canonical quotient map and
the lower horizontal arrow is given by $\phi(x) = x$.
One can easily check (e.g., use Eisenstein's criterion) that
$t^p -x$ is an irreducible polynomial in $k[x,t]$.
If we let $\fq_1 = (t)$ and $\fq_2 = (t^p -x)$, then we see that
$k(x) \to k(\fq_1)$ is an isomorphism and $k(x) \to 
k(\fq_2) \xrightarrow{\simeq} k(x^{1/p})$ is purely inseparable.
In particular, $\phi$ induces a weakly separable map between the function
fields. It follows from \cite[Proposition~3]{Yanagihara} that
$A$ is weakly normal. We just saw however that $\phi$ is not generically
separable. 

Note also that the kernels of the vertical arrows in 
~\eqref{eqn:w-normal-0} are isomorphic and the kernel on the
right is a principal ideal generated by $f(x,t) = t(t^p-x)$.
Since $f(x,t)$ a non-zero-divisor in $k[x,t]$ and lies in $A$, it
must be a non-zero-divisor in $A$. Since $k[x]$ is Cohen-Macaulay,
it follows that $A$ is Cohen-Macaulay too (see \cite[Theorem~17.3]{Matsumura}).
In particular, it is $S_2$. 

\end{exm}


Motivated by Example~\ref{example:w-normal}, we make the following 
definition.

\begin{defn}\label{defn:swn}
Let $A$ be a commutative reduced Noetherian ring with finite normalization
$B$. Let $I \subset B$ be the largest ideal lying in $A$.
We shall say that $A$ is {\sl separably weakly normal} if $B/I$ is
reduced and the induced map $A/I \to B/I$ is generically separable. 
We shall say that a Noetherian scheme $X$ is (separably) weakly normal
(resp. semi-normal) if the coordinate ring of every affine open in $X$ 
is (separably) weakly normal (resp. semi-normal).
\end{defn}

It follows from \propref{prop:weakns} that if $A$ is $S_2$ and
separably weakly normal, then it is weakly normal. 
On the other hand, Example~\ref{example:w-normal} shows that 
a weakly normal $S_2$ ring may not be separably weakly normal.
For $\Q$-algebras, semi-normality implies separably weak normality by
\propref{prop:weakns}.
The three singularity types coincide for $\Q$-algebras which are $S_2$.
Our goal in the rest of this text is to study the torsion in the
Chow group of 0-cycles on schemes which are separably weakly normal.

\begin{remk}\label{remk:Ignore}
Since we only deal with schemes which are separably weakly normal
in the rest of this text, the reader can, in principle, only read 
Definition~\ref{defn:swn} in \S~\ref{sec:WN} and move to the next subsection.
Our discussion of semi-normality and weak normality
is primarily meant to motivate the reader to 
Definition~\ref{defn:swn}, and to give a comparison of various 
non-normal singularity types which are closely related yet 
different.
\end{remk} 

We end this subsection by recalling the notion of conducting ideals 
and conducting subschemes.
Let $A$ be a reduced commutative Noetherian ring and let $B$ be a subring of
the integral closure of $A$ in its total quotient ring such that $A \subset B$.
Assume that $B$ is a finite $A$-module.
Recall that an ideal $I \subset A$ is called a conducting ideal for the 
inclusion $A \subset B$ if $I = IB$. It is clear from this definition
that $I \subset (A:B)$. Furthermore, one knows that
$(A:B)$ is the largest conducting ideal for $A \subset B$ 
(e.g., see \cite[Ex.~2.11]{HS}).
If $X$ is a reduced Noetherian scheme and $f: X' \to X$ is a 
finite birational map,
then a closed subscheme $Y \subset X$ is called a conducting subscheme
if $\sI_Y \subset \sO_X$ is a sheaf of conducting ideals for the
inclusion of sheaves of rings $\sO_X \subset f_*(\sO_{X'})$.

\subsection{The torsion theorem for separably weakly normal surfaces}
\label{sec:TS}
For the remaining part of this section, we shall identify the two 
Chow groups $\CH^{LW}_0(X)$ and $\CH_0(X)$ for curves and
surfaces using \cite[Theorem~3.17]{BK}.
To prove the Roitman torsion theorem for a separably
weakly normal projective surface, we
need the following excision result for our Chow group.
We fix an algebraically closed field $k$.

\begin{lem}\label{lem:WN-0}
Let $X$ be a reduced quasi-projective separably 
weakly normal surface over $k$ and let 
$f:\ov{X} \to X$ denote the 
normalization map. Let $Y \subset X$ denote the smallest conducting
closed subscheme with $\ov{Y} = Y \times_X \ov{X}$. Then there is an
exact sequence 
\begin{equation}\label{eqn:WN-0-1}
SK_1(\ov{X}) \oplus SK_1(Y) \to SK_1(\ov{Y}) \to \CH_0(X) \to \CH_0(\ov{X}) \to
0.
\end{equation}
\end{lem}
\begin{proof}
We have a commutative diagram of relative and birelative $K$-theory exact 
sequences:
\begin{equation}\label{eqn:WN-0-2}
\xymatrix@C.8pc{
& & K_0(X, \ov{X}, Y) \ar[d] & & & K_{-1}(X, \ov{X}, Y) \ar[d] \\ 
K_1(X) \ar[r] \ar[d] & K_1(Y) \ar[r] \ar[d] & K_0(X,Y) \ar[r] \ar[d] &
K_0(X) \ar[r] \ar[d] & K_0(Y) \ar[r] \ar[d] & K_{-1}(X,Y) \ar[d] \\
K_1(\ov{X}) \ar[r] & K_1(\ov{Y}) \ar[r] & K_0(\ov{X},\ov{Y}) \ar[r] \ar[d] &
K_0(\ov{X}) \ar[r] & K_0(\ov{Y}) \ar[r] & K_{-1}(\ov{X},\ov{Y}) \ar[d] \\
& & K_{-1}(X, \ov{X}, Y) & & & K_{-2}(X, \ov{X}, Y).}
\end{equation}

Using the Thomason-Trobaugh spectral sequence and the results of Bass 
\cite{Bass} that
the sheaves $\sK_{i, (X, \ov{X}, Y)}$ on $Y$ vanish for $i \le 0$, we see that
$K_i(X, \ov{X}, Y) = 0$ for $i \le -1$. 
It also follows from the spectral sequence that 
$K_0(X, \ov{X}, Y) \simeq H^1(Y, \sK_{1, (X, \ov{X}, Y)})$.
It follows from \cite[Theorem~0.2]{GW} that
$H^1(Y, \sK_{1, (X, \ov{X}, Y)}) \simeq 
H^1(X, {{\sI}_{Y}}/{{\sI}_{Y}^2} \otimes_{\sO_Y} {\Omega}^1_{{\ov Y}/{Y}})$.
Since $X$ is separably weakly normal, it follows from 
Definition~\ref{defn:swn} that
the map $\ov{Y} \to Y$ is a finite generically {\'e}tale map of reduced
schemes.
Hence ${\Omega}^1_{{\ov Y}/{Y}}$ has 0-dimensional support and
we conclude that 
$H^1(X, {{\sI}_{Y}}/{{\sI}_{Y}^2} \otimes_{\sO_Y} {\Omega}^1_{{\ov Y}/{Y}}) = 0$.
In particular, $K_0(X, \ov{X}, Y) = 0$. 

A diagram chase in ~\eqref{eqn:WN-0-2} shows that there is an exact sequence
\[
K_1(\ov{X}) \oplus K_1(Y) \to K_1(\ov{Y}) \to \wt{K}_0(X) \to \wt{K}_0(\ov{X})
\oplus \wt{K}_0(Y) \to \wt{K}_0(\ov{Y}),
\]
where $\wt{K}_0(-) = {\rm Ker}(K_0(-) \to H^0(-, \Z))$.
Using the map of sheaves $\sO^{\times}_X \to f_*(\sO^{\times}_{\ov{X}})$, this 
exact sequence maps to a similar unit-Pic exact sequence 
\[
U(\ov{X}) \oplus U(Y) \to U(\ov{Y}) \to \Pic(X) \to \Pic(\ov{X})
\oplus \Pic(Y) \to \Pic(\ov{Y}).
\]
Taking the kernels and using the Levine's formula 
$SK_0(Z):= {\rm Ker}(\wt{K}_0(Z) \to \Pic(Z)) 
\simeq H^2(Z, \sK_2) \simeq \CH_0(Z)$ for a reduced surface $Z$,
we get ~\eqref{eqn:WN-0-1}.
\end{proof}

\begin{lem}\label{lem:SK-1-general}
Let $Y$ be a reduced curve over $k$ and let $r \ge 0$ denote the number of
irreducible components of $Y$ which are projective over $k$.
If $Y$ is affine, then $SK_1(Y)$ is uniquely divisible. If $Y$ is projective,
then $SK_1(Y)$ is divisible, $SK_1(Y)\{l\} \simeq ({\Q_l}/{\Z_l})^{r}$
for a prime $l \neq p$ and $SK_1(Y)\{p\} = 0$.
\end{lem}
\begin{proof}
It follows from \cite[Theorem~5.3]{BPW} that 
$SK_1(Y) \simeq (k^{\times})^r \oplus V$, where $V$ is uniquely divisible.
The lemma follows directly from this isomorphism.
\end{proof}

\begin{thm}\label{thm:Fin-dim-main}
Let $X$ be a reduced projective separably weakly normal surface over an 
algebraically closed field $k$ of exponential characteristic $p$.
Then, $A^2(X)$ is a semi-abelian variety and the Abel-Jacobi map
$\rho_X: \CH_0(X)_{\deg 0} \to A^2(X)$ is an isomorphism on the torsion
subgroups.
\end{thm}
\begin{proof}
We can find a finite collection of reduced complete intersection Cartier
curves $\{C_1, \cdots , C_r\}$ on $X$ such that the induced map
of algebraic groups $\stackrel{r}{\underset{i =1} \prod} \Pic^0(C_i) \to
A^2(X)$ is surjective (see, for instance, \cite[(7.1), p.~657]{ESV}). 
In characteristic zero, we can further assume by \cite[Corollary~2.5]{CGM}
that each of these curves is weakly normal.
Since a surjective morphism of smooth connected algebraic groups
restricts to a surjective map on their unipotent and semi-abelian parts
and since $\Pic^0(C)$ is semi-abelian if $C$ is a weakly normal curve
(easy to check), it follows that $A^2(X)$ is a semi-abelian variety in
characteristic zero.

In characteristic $p \ge 2$, the surjections
$\stackrel{r}{\underset{i =1} \prod} \Pic^0(C_i) \surj A^2(X) \surj J^2(X)$  
and \cite[Lemma~2.7]{Krishna-1} together imply that 
the induced maps $(\CH^2(X)_{\deg 0})_{\rm tors} \to A^2(X)_{\rm tors}
\to J^2(X)_{\rm tors}$ are also surjective. 
Since $A^2_{\rm unip}(X)$ is a $p$-primary torsion group of
bounded exponent, the theorem will follow if we prove, in any
characteristic, that the composite map 
$\rho^{\rm semi}_X: \CH_0(X)_{\deg 0} \to A^2(X) \to J^2(X)$ is 
injective on the torsion subgroup.

It follows from \cite[Theorem~15]{Mallick} that the map
$\CH_0(X)_{\deg 0}\{l\} \to J^2(X)\{l\}$ 
is an isomorphism for every prime $l \neq p$.
This also follows immediately from \thmref{thm:K-2-coh-0} and
the proof of \thmref{thm:Cyc-iso}.
We can thus assume that $p \ge 2$ and $l = p$.

Let $f: \ov{X} \to X$ be the normalization
and consider the commutative diagram
\begin{equation}\label{eqn:SK-1-curve-1*} 
\xymatrix@C1pc{
\CH_0(X)_{\deg 0}\{p\} \ar[r]^-{f^*} \ar[d] & 
\CH_0(\ov{X})_{\deg 0}\{p\} \ar[d] \\
J^2(X)\{p\} \ar[r] & J^2(\ov{X})\{p\}.}
\end{equation}

The right vertical arrow is an isomorphism by \cite[Theorem~1.6]{KS}.
Thus, it suffices to prove the stronger assertion that the map
$f^*: \CH_0(X)\{p\} \to \CH_0(\ov{X})\{p\}$ is an isomorphism.
It is enough to show that $A := {\rm Ker}(\CH_0(X) \surj \CH_0(\ov{X}))$
is uniquely $p$-divisible.

It follows from
Lemma~\ref{lem:WN-0} that there is an exact sequence
\[
SK_1(\ov{X}) \oplus SK_1(Y) \to SK_1(\ov{Y}) \to A \to 0.
\]
Following the notations of Lemma~\ref{lem:WN-0}, it follows from 
\lemref{lem:SK-1-general} that $SK_1(Y)$ and $SK_1(\ov{Y})$ are
uniquely $p$-divisible. Furthermore, it follows from
\cite[Theorem~5.6]{Krishna-1} that $SK_1(\ov{X}) \otimes {\Q_p}/{\Z_p} = 0$.
An elementary argument now shows that $A$ must be uniquely $p$-divisible.
This finishes the proof.
\end{proof}

\begin{cor}\label{cor:fin-dim-cor}
Let $X$ be a reduced projective separably weakly normal surface over 
$\ov{\F}_p$.
Then $\CH_0(X)_{\deg 0}$ is finite-dimensional. That is, the Abel-Jacobi map
$\rho_X: \CH_0(X)_{\deg 0} \to A^2(X)$ is an isomorphism.
\end{cor}
\begin{proof}
In view of \thmref{thm:Fin-dim-main}, it suffices to show that for
any reduced projective scheme $X$ of dimension $d \ge 1$ over $\ov{\F}_p$,
the group $\CH^{LW}_0(X)_{\deg 0}$ is a torsion abelian group.

Given $\alpha \in \CH^{LW}_0(X)_{\deg 0}$, we can find a reduced Cartier curve
$C \subseteq X$ such that $\alpha$ lies in the image of the push-forward
map $\Pic^0(C) \to \CH^{LW}_0(X)_{\deg 0}$. It is therefore enough to show that
$\Pic^0(C)$ is torsion. But this is a special case of the more general
fact that $G(\ov{\F}_p)$ is torsion whenever $G$ is a 
smooth commutative algebraic group over $\ov{\F}_p$.
\end{proof}

\section{Bloch's torsion theorem for 0-cycles 
with modulus}\label{sec:TEC}
We continue with our assumption that $k$ is algebraically closed with
exponential characteristic $p$. Let $X$ be a smooth projective scheme of
dimension $d \ge 1$ over $k$ and let $D \subset X$ be an effective Cartier
divisor on $X$. We shall prove \thmref{thm:Intro-1} in this section.

\subsection{Relative {\'e}tale cohomology}\label{sec:Et-coh}
Given an {\'e}tale sheaf $\sF$ on $\Sch_k$ and a finite map $f:Y \to X$
in $\Sch_k$, let $\sF_{(X,Y)} := {\rm Cone}(\sF|_{X} \to f_*(\sF|_Y))[-1]$
be the chain complex of {\'e}tale sheaves on $X$. The exactness of
$f_*$ on {\'e}tale sheaves implies that there is a long exact
sequence of {\'e}tale (hyper)cohomology groups
\begin{equation}\label{eqn:rel-et}
0 \to H^0_{\etl}(X, \sF_{(X,Y)}) \to H^0_{\etl}(X, \sF) \to
H^0_{\etl}(Y, \sF) \to H^1_{\etl}(X, \sF_{(X,Y)}) \to \cdots .
\end{equation}

If $Y \inj X$ is a closed immersion with complement $j: U \inj X$,
then one checks immediately from the above definition that
$H^*_{\etl}(X,  \sF_{(X,Y)})$ is canonically isomorphic to 
$H^*_{\etl}(X, j_{!}(\sF|_U))$. We conclude that if $X$ is projective over $k$
and $Y \inj X$ is closed,
then $H^*_{\etl}(X,  \sF_{(X,Y)})$ is same as the {\'e}tale 
cohomology with compact
support $H^*_{c, \etl}(X \setminus Y, \sF)$ of $X \setminus Y$.

Let us now consider an abstract blow-up diagram in $\Sch_k$:
\begin{equation}\label{eqn:rel-et-0}
\xymatrix@C1pc{
Y' \ar[r]^{f'} \ar[d]_{\pi'} & X' \ar[d]^{\pi} \\
Y \ar[r]_{f} & X.}
\end{equation}
This is a Cartesian square in which the horizontal arrows are closed
immersions, $\pi$ is a proper morphism such that 
$X' \setminus Y' \xrightarrow{\simeq} X \setminus Y$. 
The proper base change theorem for the {\'e}tale cohomology implies that
for a torsion constructible {\'e}tale sheaf $\sF$ on $\Sch_k$,
the cohomology groups $H^*_{\etl}(-, \sF)$ satisfy the $cdh$-descent.
In particular, the canonical map 
\begin{equation}\label{eqn:rel-et-1}
\pi^*: H^i_{\etl}(X, \sF_{(X,Y)}) \to H^i_{\etl}(X', \sF_{(X',Y')})
\end{equation}
is an isomorphism for every $i \ge 0$.
We shall write the {\sl relative {\'e}tale cohomology groups}
$H^*_{\etl}(X, \sF_{(X,Y)})$ for a closed immersion $Y \inj X$ 
as $H^*_{\etl}(X|Y, \sF)$.

\subsection{A weak Lefschetz type theorem for the double}
\label{sec:WLD}
We now come back to our situation of $X$ being a smooth projective scheme
of dimension $d \ge 1$ over $k$ and $D \subset X$ an effective Cartier divisor.
Let $\{E_1, \cdots , E_r\}$ be the set of irreducible components of 
$D_{\rm red}$.
If $d \ge 3$, we choose, as in \S~\ref{sec:Lef}, a closed embedding 
$X \inj \P^N_k$ and a smooth hypersurface section 
$\tau:Y = X \cap H_1 \inj X$ such that $Y$ is not contained in $D$
and $Y \cap E_i$ is integral for every $1 \le i \le r$.
We set $F = D \cap Y$. We shall use these notations throughout the
rest of this section.

Given a prime-to-$p$ integer $n$, it is clear from the definition of the 
relative {\'e}tale cohomology and its $cdh$-descent
(see \S~\ref{sec:Et-coh}) associated to the Cartesian 
square ~\eqref{eqn:rel-et-2}
(see \lemref{lem:rel-et-2*}) that the pull-back maps
via the closed immersions $\iota_\pm: X \inj S_X$ induce, for each
$i \ge 0$, a split exact sequence of {\'e}tale cohomology
\begin{equation}\label{eqn:BK-Main-et}
0 \to H^i_{\etl}(X|D, \mu_{n}(j)) \xrightarrow{p_{+, *}} 
H^i_{\etl}(S_X, \mu_{n}(j))  
\xrightarrow{\iota^*_-} H^i_{\etl}(X, \mu_{n}(j))  \to 0.
\end{equation}
Here, the splitting of $\iota^*_-$ is given by the pull-back
$\Delta^*: H^i_{\etl}(X, \mu_{n}(j)) \to  H^i_{\etl}(S_X, \mu_{n}(j))$.

Let us next recall from Gabber's construction \cite{Fujiwara} of the
Gysin morphism for {\'e}tale cohomology  (see, also
\cite[Definition~2.1]{Navarro})
that the regular closed embeddings $\tau:Y \inj X$ and $\tau_1:S_Y \inj S_X$ 
induce Gysin morphisms $\tau_*: H^i_{\etl}(Y, \mu_{n}(j)) \to
H^{i+2}_{\etl}(X, \mu_{n}(j+1))$ and $\tau_{1,*}: H^i_{\etl}(S_Y, \mu_{n}(j)) \to
H^{i+2}_{\etl}(S_X, \mu_{n}(j+1))$ for $i \ge 0$. Furthermore, it follows
from the Cartesian square ~\eqref{eqn:PF-alb-sing-1} and
\cite[Corollary~2.12]{Navarro} (see also
\cite[Proposition~1.1.3]{Fujiwara}) that the pull-back
via the closed immersions $\iota_\pm: X \inj S_X$ induces a commutative diagram
\begin{equation}\label{eqn:PF-alb-sing-3} 
\xymatrix@C1pc{
H^i_{\etl}(S_Y, \mu_{n}(j)) \ar[r]^-{\iota^*_\pm} \ar[d]_{\tau_{1,*}} & 
H^i_{\etl}(Y, \mu_{n}(j)) \ar[d]^{\tau_*} \\
H^{i+2}_{\etl}(S_X, \mu_{n}(j+1)) \ar[r]_-{\iota^*_\pm} & 
H^{i+2}_{\etl}(X, \mu_{n}(j+1)).}
\end{equation}

\begin{lem}\label{lem:Lef-etale}
If $d \ge 3$, then the Gysin maps $\tau_*: H^{2d-3}_{\etl}(Y, \mu_{n}(d-1)) \to
H^{2d-1}_{\etl}(X, \mu_{n}(d))$ and 
$\tau_{1,*}: H^{2d-3}_{\etl}(S_Y, \mu_{n}(d-1)) 
\to H^{2d-1}_{\etl}(S_X, \mu_{n}(d))$ are isomorphisms.
\end{lem}
\begin{proof}
The first isomorphism is a well known consequence of the weak Lefschetz theorem 
for {\'e}tale cohomology. The main problem is to prove the second
isomorphism. 
Since $k$ is algebraically closed, we shall replace $\mu_n$ by the
constant sheaf $\Lambda = {\Z/n}$.

Recall from \cite[\S~2]{Deligne} that the line bundle $\sO_X(Y)$
(which we shall write in short as $\sO(Y)$) on $X$ 
has a canonical class $[\sO(Y)] \in H^1_{Y, \etl}(X, \G_m)$ and its image
via the boundary map $H^1_{Y, \etl}(X, \G_m) \to H^2_{Y, \etl}(X, \Lambda)$ is 
Deligne's localized Chern class $c_1(Y)$. 
Here, $H^*_{Y, \etl}(X, -)$ denotes the {\'e}tale cohomology with support
in $Y$. On the level of the derived 
category $D^{+}(Y, \Lambda)$, this Chern class is given in terms
of the map $c_1(Y): \Lambda \to \tau^{!}\Lambda(1)[2]$.
Using this Chern class, Gabber's Gysin morphism 
$\tau_*: H^*_{\etl}(Y, \Lambda(j)) \to H^{*+2}_{\etl}(X, \Lambda(j+1))$ is the one
induced on the cohomology by the composite map 
$\tau_*: \tau_*(\Lambda_Y) \to \tau_* \tau^{!}(\Lambda_Y(1)[2]) \to
\Lambda_X(1)[2]$ in $D^{+}(X, \Lambda)$ .

Corresponding to the Cartesian square ~\eqref{eqn:PF-alb-sing-1},
we have $\pi^*(\sO(S_Y)) = \sO(Y \amalg Y)$ and hence we
have a commutative diagram
\begin{equation}\label{eqn:Lef-etale-0}
\xymatrix@C1pc{
H^1_{S_Y,\etl}(S_X, \G_m) \ar[r]^{\partial} \ar[d]_{\pi^*} & 
H^2_{S_Y, \etl}(S_X, \Lambda) \ar[d]^{\pi^*} \\
H^1_{Y \amalg Y,\etl}(X \amalg X, \G_m) \ar[r]_{\partial} & 
H^2_{Y \amalg Y,\etl}(X \amalg X, \Lambda).}
\end{equation}

We thus have commutative diagrams
\begin{equation}\label{eqn:Lef-etale-1}
\xymatrix@C1pc{
\tau_{1, *}(\Lambda_{S_Y}) \ar[r]^-{\tau_{1,*}} \ar[d]_{\pi^*} &
\Lambda_{S_X}(1)[2] \ar[d]^{\pi^*} &
\tau_*(\Lambda_{Y}) \ar[r]^{\tau_*} \ar[d]_{\iota^*_F} & 
\Lambda_{X}(1)[2] \ar[d]^{\iota^*_D} \\
\pi_*\psi_*(\Lambda_{Y \amalg Y}) \ar[r]_-{\pi_* \psi_*} & 
\pi_*\Lambda_{X \amalg X}(1)[2] & \tau_*\iota_{F,*}(\Lambda_F) \ar[r]_-{\tau_*} &
\iota_*(\Lambda_D(1)[2])}
\end{equation}
in $D^{+}(S_X, \Lambda)$ and $D^{+}(X, \Lambda)$.

We next observe that the canonical map $\Lambda_{X \amalg X} \to 
\Lambda_X \oplus \Lambda_X$, induced by the inclusions of the two
components, is an isomorphism.
We thus have a sequence of maps
$\Lambda_{S_X} \xrightarrow{\pi^*} \pi_*(\Lambda_{X \amalg X}) \simeq
\Lambda_{X_+} \oplus \Lambda_{X_-} \to \Lambda_D$, 
where the last map is the difference of two restrictions
$\Lambda_{X_\pm} \to \iota_*(\Lambda_D)$.
Furthermore, it is easy to check that the sequence
\[
0 \to \Lambda_{S_X}(j) \xrightarrow{\pi^*} \pi_*(\Lambda_{X \amalg X}(j)) \to
\Lambda_D(j) \to 0
\]
is exact. A combination of this with ~\eqref{eqn:Lef-etale-1} yields
a commutative diagram of exact triangles in $D^{+}(S_X, \Lambda)$:
\begin{equation}\label{eqn:Lef-etale-2}
\xymatrix@C1pc{
0 \ar[r] & \tau_{1, *}(\Lambda_{S_Y}) \ar[r]^-{\pi^*} \ar[d]_{\tau_{1,*}} &
\pi_*\psi_*(\Lambda_{Y \amalg Y}) \ar[d]^-{\pi_* \psi_*} \ar[r] & 
\tau_{1,*}(\Lambda_F) \ar[d]^{\tau_{1,*}} \ar[r] & 0 \\
0 \ar[r] & \Lambda_{S_X}(1)[2] \ar[r]_-{\pi^*} & \pi_*(\Lambda_{X \amalg X}(1)[2])
\ar[r] & \Lambda_D(1)[2] \ar[r] & 0.}
\end{equation}

Since all the underlying maps in ~\eqref{eqn:PF-alb-sing-1} are finite,
we obtain a commutative diagram of long exact sequence of cohomology
groups
\begin{equation}\label{eqn:Lef-etale-3}
\xymatrix@C.3pc{
H^{2d-4}_{\etl}(Y \amalg Y, \Lambda(d-1)) \ar[r] \ar[d] &
H^{2d-4}_{\etl}(F, \Lambda(d-1)) \ar[r] \ar[d] &
H^{2d-3}_{\etl}(S_Y , \Lambda(d-1)) \ar[r] \ar[d]^{\tau_{1,*}} &
H^{2d-3}_{\etl}(Y \amalg Y, \Lambda(d-1)) \ar[r] \ar[d] & 0 \\
H^{2d-2}_{\etl}(X \amalg X, \Lambda(d)) \ar[r] &
H^{2d-2}_{\etl}(D, \Lambda(d)) \ar[r] &
H^{2d-1}_{\etl}(S_X , \Lambda(d)) \ar[r] &
H^{2d-1}_{\etl}(X \amalg X, \Lambda(d)) \ar[r] & 0.}
\end{equation}

Since $d \ge 3$, it follows from the weak Lefschetz theorem for {\'e}tale
cohomology (see \cite[Theorem~VI.7.1]{Milne}) 
that the vertical arrow on the left end of ~\eqref{eqn:Lef-etale-3}
is surjective and the one on the right end is an isomorphism. 
We next note that the inclusion $F = D \cap H_1 \inj D$ induces 
a bijection between the irreducible components of $F$ and $D$ 
by our choice of the hypersurface $H_1$ (see \S~\ref{sec:Lef}).
It follows from this, together with the isomorphism
$H^{2d-2}_{\etl}(D, \Lambda) \xrightarrow{\simeq} H^{2d-2}_{\etl}(D^N, \Lambda)$ and
\cite[Lemma~VI.11.3]{Milne}, that the second vertical arrow from the left 
in ~\eqref{eqn:Lef-etale-3} is
an isomorphism. A diagram chase now shows that $\tau_{1,*}$ is an
isomorphism. This proves the lemma.
\end{proof}

\begin{remk}\label{remk:WLD-*}
We should warn here that \lemref{lem:Lef-etale} proves an
analogue of the weak Lefschetz theorem only for a specific {\'e}tale
cohomology group. We do not expect this to be true for other
cohomology groups in general and it will depend critically on $D$.
\end{remk}

\subsection{The torsion in Chow group with modulus and relative
{\'e}tale cohomology}\label{sec:Tor-thm}
Let $D \subset X$ be an effective Cartier divisor on a smooth scheme $X$
as above.
Recall from \cite[\S~4, 5]{BK} that there are maps
$p_{\pm, *}: \CH_0(X|D) \to \CH_0(S_X)$ and $\iota^*_{\pm}: \CH_0(S_X) \to 
\CH_0(X)$, where $p_{\pm, *}([x]) = \iota_{\pm, *}([x])$ and
$\iota^*_{\pm}$ is induced by the projection map
$\sZ_0(S_X \setminus D) = \sZ_0(X_+ \setminus D) \oplus  
\sZ_0(X_- \setminus D) \surj \sZ_0(X_\pm \setminus D)$. 
It follows at once from this description of
the projection maps $\iota^*_{\pm}$ that there is a
commutative diagram
\begin{equation}\label{eqn:PF-alb-sing-2} 
\xymatrix@C1pc{
\CH_0(S_Y)_{\deg 0} \ar[r]^-{\iota^*_\pm} \ar[d]_{\tau_{1,*}} & 
\CH_0(Y)_{\deg 0} \ar[d]^{\tau_*} \\
\CH_0(S_X)_{\deg 0} \ar[r]_-{\iota^*_\pm} & \CH_0(X)_{\deg 0}.}
\end{equation}

We shall use the following decomposition theorem from \cite[Theorem~7.1]{BK}.
\begin{thm}\label{thm:BK-Main}
The projection map 
$\Delta: S_X \to X$ induces a flat pull-back
$\Delta^*: \CH_0(X) \to \CH_0(S_X)$ such that $\iota^*_\pm \circ \Delta^*
= {\rm Id}_{\CH_0(X)}$. Moreover, there is a split exact sequence
\begin{equation}\label{eqn:BK-Main-0}
0 \to \CH_0(X|D) \xrightarrow{p_{+, *}} \CH_0(S_X) \xrightarrow{\iota^*_-}
\CH_0(X) \to 0.
\end{equation}
\end{thm}

\begin{thm}\label{thm:Torsion-Bloch}
Let $X$ be a smooth projective scheme of dimension $d \ge 1$
over an algebraically closed field $k$ of exponential characteristic $p$.
Then for any prime $l \neq p$, there is an isomorphism 
\[
\lambda_{X|D}: 
\CH_0(X|D)\{l\} \xrightarrow{\simeq}  H^{2d-1}_{\etl}(X|D, {\Q_l}/{\Z_l}(d)).
\]
\end{thm}
\begin{proof}
If $D = \emptyset$, we take $\lambda_{X|D}$ to be the
isomorphism $\lambda_X: \CH_0(X)\{l\} \xrightarrow{\simeq}  
H^{2d-1}_{\etl}(X, {\Q_l}/{\Z_l}(d))$ given by Bloch \cite[\S~2]{Bloch-tor}.

We now assume $D \neq \emptyset$ and let $S_X$ be the double of $X$ along $D$. 
We shall first prove by induction on $d$ that there exists
an isomorphism 
\begin{equation}\label{eqn:extra}
\lambda_{S_X}: \CH_0(S_X)\{l\} \xrightarrow{\simeq}  
H^{2d-1}_{\etl}(S_X, {\Q_l}/{\Z_l}(d))
\end{equation}
such that $\iota^*_\pm \circ \lambda_{S_X} = \lambda_X \circ \iota^*_\pm$.

When $d =1$,
it follows easily from the Kummer sequence and \cite[Proposition~1.4]{LW}
that there is a natural isomorphism
$H^1_{\etl}(C, \mu_n(1)) \xrightarrow{\simeq} {_{n}{\CH_0(C)}}$ for any
reduced curve $C$ over $k$ and any integer $n \ge 1$ prime to $p$.
The naturality of this isomorphism proves our assertion.
Note that this isomorphism coincides with that of Bloch when
$C$ is smooth, as one directly checks 
(or see \cite[Proposition~3.6]{Bloch-tor}).


We next assume $d = 2$. In this case, we have shown in 
\thmref{thm:K-2-coh-0} that there is a homomorphism
$\tau_Y: H^3_{\etl}(Y, \mu_n(2)) \to {_{n}{\CH_0(Y)}}$ for any reduced 
quasi-projective surface $Y$ over $k$ and any integer $n$ prime to $p$. 
We claim that the diagram
\begin{equation}\label{eqn:dim-2-com}
\xymatrix@C1pc{
H^3_{\etl}(S_X, \mu_n(2)) \ar[r]^-{\tau_{S_X}} \ar[d]_{\iota^*_-} &  
{_{n}{\CH_0(S_X)}} \ar[d]^{\iota^*_-} \\
H^3_{\etl}(X, \mu_n(2)) \ar[r]_-{\tau_X} &  
{_{n}{\CH_0(X)}}}
\end{equation}
commutes.

To prove this, recall from the construction of $\tau_Y$
in \S~\ref{sec:K-2-surf} that there are isomorphisms
$H^3_{\etl}(Y, \mu_n(2)) \xrightarrow{\simeq} \H^2(Y, \sK^{\bullet}_2)$ and
$\CH_0(Y)  \xrightarrow{\simeq} \CH^{LW}_0(X) \xrightarrow{\simeq} 
H^2(Y, \sK_2)$ (see \cite[Theorem~3.17]{BK})
which are clearly functorial for the normalization map $Y^N \to Y$.
Since $X_\pm$ are two disjoint components of the normalization
$S^N_X$, we see that these isomorphisms are functorial for the
inclusions $\iota_\pm: X_\pm \inj S_X$.

For a surface $Y$, the map $\tau_Y$ is then
the natural map $\H^2(Y, \sK^{\bullet}_2) \to {_{n}{H^2(Y, \sK_2)}}$
obtained via the exact sequence $\H^2(Y, \sK^{\bullet}_2) \to 
H^2(Y, \sK_2) \xrightarrow{n} H^2(Y, \sK_2)$.
The commutativity of ~\eqref{eqn:dim-2-com} then follows immediately
from the naturality of the complex of Zariski sheaves on 
$\sK_2[-1] \to \sK^{\bullet}_2 \to \sK_2$ on $\Sch_k$.
This proves the claim.

\thmref{thm:Cyc-iso} says that $\tau_Y:H^{3}_{\etl}(Y, {\Q_l}/{\Z_l}(2))
\xrightarrow{\simeq} \CH_0(Y)\{l\}$ is an isomorphism 
on the limit for every prime $l \neq p$.
We let $\lambda_Y: \CH_0(Y)\{l\} \to H^{3}_{\etl}(Y, {\Q_l}/{\Z_l}(2))$
be the negative of the inverse of this isomorphism.
When $Y$ is smooth over $k$, it follows from 
\lemref{lem:Bloch-surface} that $\lambda_Y$ agrees with Bloch's isomorphism.
It follows then from ~\eqref{eqn:dim-2-com} that 
\begin{equation}\label{eqn:dim-2-com*}
\iota^*_\pm \circ \lambda_{S_X} = \lambda_X \circ \iota^*_\pm.
\end{equation}

We now assume $\dim(X) \ge 3$. We choose a closed embedding $X \inj \P^N_k$
and hypersurfaces $H_1, \cdots, H_{d-2}$ as in ~\eqref{eqn:Bertini}. 
We continue to use the notations of \propref{prop:PF-alb-sing}. 
We set $Y = X_1 := X \cap H_1$ and assume by induction that
there is an isomorphism $\lambda_{S_Y}: \CH_0(S_Y)\{l\} \xrightarrow{\simeq}  
H^{2d-3}_{\etl}(S_Y, {\Q_l}/{\Z_l}(d-1))$ such that the diagram
\begin{equation}\label{eqn:Torsion-Bloch}
\xymatrix@C1pc{
\CH_0(S_Y)\{l\} \ar[r]^-{\lambda_{S_Y}} \ar[d]_{\iota^*_\pm} & 
H^{2d-3}_{\etl}(S_Y, {\Q_l}/{\Z_l}(d-1)) \ar[d]^{\iota^*_\pm} \\
\CH_0(Y)\{l\} \ar[r]_-{\lambda_{Y}} &
H^{2d-3}_{\etl}(Y, {\Q_l}/{\Z_l}(d-1))} 
\end{equation}
commutes.

We now consider the diagram
\begin{equation}\label{eqn:Torsion-Bloch-0}
\xymatrix@C1pc{
\CH_0(S_Y)\{l\} \ar[r]^-{\lambda_{S_Y}} \ar[d]_{\tau_{1,*}} &  
H^{2d-3}_{\etl}(S_Y, {\Q_l}/{\Z_l}(d-1)) \ar[d]^{\tau_{1,*}} \\
\CH_0(S_X)\{l\} \ar@{.>}[r]_-{\lambda_{S_X}} & 
H^{2d-1}_{\etl}(S_X, {\Q_l}/{\Z_l}(d)).}
\end{equation}

It follows from \propref{prop:PF-alb-sing} that the left vertical arrow is
an isomorphism and \lemref{lem:Lef-etale} says that the right vertical
arrow is an isomorphism. Since $\lambda_{S_Y}$ is an isomorphism too,
it follows that there is a unique isomorphism
$\lambda_{S_X}: \CH_0(S_X)\{l\} \xrightarrow{\simeq} 
H^{2d-1}_{\etl}(S_X, {\Q_l}/{\Z_l}(d))$ such that ~\eqref{eqn:Torsion-Bloch-0}
commutes.

We have to check that $\iota^*_\pm \circ \lambda_{S_X} = \lambda_{X} \circ 
\iota^*_{\pm}$. For this, we consider the diagram
\[
\xymatrix@C.001pc{
\CH_0(S_Y)\{l\} \ar[dd]_-{\lambda_{S_Y}} \ar[dr]_{\tau_{1,*}} 
\ar[rr]^{\iota^*_\pm} & & \CH_0(Y))\{l\} \ar[dd]^>>>>>>>{\lambda_{Y}} \ar[dr]^{\tau_*} & \\
& \CH_0(S_X))\{l\} \ar[rr]_<<<<<<<<<<<{\iota^*_\pm} \ar[dd]^>>>>>>>{\lambda_{S_X}} & & 
\CH_0(X))\{l\} \ar[dd]^{\lambda_X} \\
H^{2d-3}_{\etl}(S_Y, {\Q_l}/{\Z_l}(d-1)) \ar[rr]^<<<<<<<<<<<{\iota^*_\pm} 
\ar[dr]_{\tau_{1,*}} & &
H^{2d-3}_{\etl}(Y, {\Q_l}/{\Z_l}(d-1)) \ar[dr]_{\tau_*} & \\
&  H^{2d-1}_{\etl}(S_X, {\Q_l}/{\Z_l}(d)) \ar[rr]_-{\iota^*_\pm} & &
H^{2d-1}_{\etl}(X, {\Q_l}/{\Z_l}(d)).}
\]

We need to show that the front face of this cube commutes. Since
$\tau_{1,*}:\CH_0(S_Y)\{l\} \to \CH_0(S_X)\{l\}$ is an isomorphism, it 
suffices to show that all other faces of the cube commute.
The top face commutes by ~\eqref{eqn:PF-alb-sing-2} and
the bottom face commutes by ~\eqref{eqn:PF-alb-sing-3}.
The left face commutes by ~\eqref{eqn:Torsion-Bloch-0} and the right face
commutes by \cite[Proposition~3.3]{Bloch-tor}. Finally, the back face commutes
by ~\eqref{eqn:Torsion-Bloch} and we are done.

To finish the proof of the theorem, we use ~\eqref{eqn:BK-Main-et} and
\thmref{thm:BK-Main} and consider the diagram
of split exact sequences:
\begin{equation}\label{eqn:Torsion-Bloch-2}
\xymatrix@C1pc{
0 \ar[r] & \CH_0(X|D)\{l\} \ar[r]^-{p_{+,*}} \ar@{-->}[d] &
\CH_0(S_X)\{l\} \ar[r]^-{\iota^*_-} \ar[d]^{\lambda_{S_X}} &
\CH_0(X)\{l\} \ar[d]^{\lambda_X} \ar[r] & 0 \\
0 \ar[r] & H^{2d-1}_{\etl}(X|D, {\Q_l}/{\Z_l}(d)) \ar[r]_-{p_{+,*}} &
H^{2d-1}_{\etl}(S_X, {\Q_l}/{\Z_l}(d)) \ar[r]_-{\iota^*_-} &
H^{2d-1}_{\etl}(X, {\Q_l}/{\Z_l}(d)) \ar[r] & 0.}
\end{equation}

It follows from ~\eqref{eqn:extra} that the square on the right is commutative.
Moreover, the maps $\lambda_{S_X}$ and $\lambda_X$ are both isomorphisms.
We conclude that there is an isomorphism
$\lambda_{X|D}: \CH_0(X|D)\{l\} \xrightarrow{\simeq} 
H^{2d-1}_{\etl}(X|D, {\Q_l}/{\Z_l}(d))$.
\end{proof}

\subsection{Applications}\label{sec:Appl}
We now deduce two applications of \thmref{thm:Torsion-Bloch}.
Since the {\'e}tale cohomology of $\mu_n(j)$ is nil-invariant whenever
$(n,p) =1$, it follows from ~\eqref{eqn:rel-et} that
$H^*_{\etl}(X|D, \mu_n(j)) \xrightarrow{\simeq} 
H^*_{\etl}(X|D_{\rm red}, \mu_n(j))$. 
Using \thmref{thm:Torsion-Bloch}, we therefore obtain the following result 
about the prime-to-$p$ torsion in the Chow group with modulus.

\begin{thm}\label{thm:nil-inv} 
Let $X$ be a smooth projective scheme of dimension $d \ge 1$ over $k$ 
let $D \subset X$ be an effective Cartier divisor.
Then, the restriction map ${_{n}{\CH_0(X|D)}} \to {_{n}{\CH_0(X|D_{\rm red})}}$ 
is an isomorphism for every integer $n$ prime to $p$.
\end{thm}

Our second application is the following extension of \thmref{thm:Cyc-iso} 
to the cohomology of relative $K_2$-sheaf on a smooth surface.
Recall that for a closed immersion $Y \inj X$ in $\Sch_k$, the
relative $K$-theory sheaf $\sK_{i, (X,Y)}$ is the Zariski sheaf on $X$
associated to the presheaf $U \mapsto K_i(U,Y\cap U)$.

\begin{thm}\label{thm:modulus-main-1}
Let $X$ be a smooth projective surface over an algebraically closed
field $k$ of exponential characteristic $p$ and let $l \neq p$ be a
prime. Let $D \subset X$ be an effective Cartier divisor. Then, the following 
hold.
\begin{enumerate}
\item
$H^1(X, \sK_{2, (X,D)})\otimes_{\Z} {\Q_l}/{\Z_l} = 0$.
\item
$H^2(X, \sK_{2, (X,D)})\{l\} \simeq H^3_{\etl}(X|D, {\Q_l}/{\Z_l}(2))$.
\end{enumerate}
\end{thm}
\begin{proof}
It follows from \thmref{thm:Torsion-Bloch} that
$H^3_{\etl}(X|D, {\Q_l}/{\Z_l}(2)) \simeq \CH_0(X|D)\{l\}$. On the
other hand, a combination of \cite[Theorem~1.7]{BK} and 
\cite[Lemma~2.1]{Krishna-3} shows that there is a canonical isomorphism
$\CH_0(X|D) \xrightarrow{\simeq} H^2(X, \sK_{2, (X,D)})$.
This proves (2).

To prove (1), we consider the
following commutative diagram for any integer $n \in k^{\times}$. 

\begin{equation}\label{eqn:rel-et-3}
\xymatrix@C.8pc{
& 0 \ar[d] & 0 \ar[d] & 0 \ar[d] & \\
0 \ar[r] & H^1(S_X, \sK_{2, (S_X, X_-)})\otimes_{\Z} {\Z}/{n} 
\ar[r] \ar[d] & H^3_{\etl}(S_X|X_-, \mu_n(2)) \ar[r] \ar[d] &
{_{n}{\CH_0(X|D)}} \ar[r] \ar[d] & 0 \\
0 \ar[r] & H^1(S_X, \sK_{2, S_X})\otimes_{\Z} {\Z}/{n} \ar[r] \ar[d] &
H^3_{\etl}(S_X, \mu_n(2)) \ar[r] \ar[d] &
{_{n}{\CH_0(S_X)}} \ar[r] \ar[d] & 0 \\
0 \ar[r] & H^1(X, \sK_{2, X})\otimes_{\Z} {\Z}/{n} \ar[r] \ar[d] &
H^3_{\etl}(X, \mu_n(2)) \ar[r] \ar[d] &
{_{n}{\CH_0(X)}}  \ar[r] \ar[d] & 0 \\
& 0 & 0 & 0 &}
\end{equation}

The columns on the left and in the middle are exact by the splitting
$\Delta \circ \iota_- = {\rm Id}_X$ in ~\eqref{eqn:rel-et-2}. The column on the
right is exact by \thmref{thm:BK-Main}. 
The middle and the bottom rows are exact by \thmref{thm:K-2-coh-0}.
It follows that the top row is also exact.
Using the isomorphism $\iota^*_+: H^3_{\etl}(S_X|X_-, \mu_n(2))
\xrightarrow{\simeq} H^3_{\etl}(X|D, \mu_n(2))$ (see ~\eqref{eqn:rel-et-1})
and taking the direct limit of the terms in ~\eqref{eqn:rel-et-3} with 
respect to the direct system $\{{\Z}/{l^n}\}_n$,
we obtain a short exact sequence
\begin{equation}\label{eqn:rel-et-3**}
0 \to H^1(S_X, \sK_{2, (S_X, X_-)})\otimes_{\Z} {\Q_l}/{\Z_l} 
\to H^3_{\etl}(X|D, {\Q_l}/{\Z_l}(2)) \to 
\CH_0(X|D)\{l\} \to 0.
\end{equation}

Using ~\eqref{eqn:rel-et-3**} 
\thmref{thm:Torsion-Bloch}, we conclude that
$H^1(S_X, \sK_{2, (S_X, X_-)}) \otimes_{\Z} {\Q_l}/{\Z_l} = 0$ for
every prime $l \neq p$.
To finish the proof, it suffices now to show that the
pull-back map $\iota^*_+: H^1(S_X, \sK_{2, (S_X, X_-)}) \to
H^1(X, \sK_{2, (X, D)})$ is surjective.

Given an open subset $W \subset D$, let $U = S_X \setminus (D \setminus W)$
be the open subset of $S_X$. Let $\sK_{i, (S_X, X_-, D)}$ be the Zariski 
sheaf on $D$ associated to the presheaf $W \mapsto 
K_i(U, X_+ \cap U, X_- \cap U) =
{\rm hofib}((K(U, X_- \cap U) \xrightarrow{i^*_+} K(X_+ \cap U, D \cap U))$ 
(see \cite[Proposition~A.5]{PW}). There is an exact sequence of 
$K$-theory sheaves
\[
\nu_{*}(\sK_{2, (S_X, X_-, D )}) \to \sK_{2, (S_X, X_-)} \to
\sK_{2, (X_+, D)} \to \nu_{*}(\sK_{1, (S_X, X_-, D)}),
\]
where $\nu:D \inj S_X$ is the inclusion.
We have $\sK_{1, (S_X, X_-, D)} = {\sI_D}/{\sI^2_D} \otimes_D \Omega^1_{D/X}$
by \cite[Theorem~1.1]{GW} and the latter term is zero.
We thus get an exact sequence
\[
\nu_{*}(\sK_{2, (S_X, X_-, D)} ) \to  \sK_{2, (S_X, X_-)} \to
\sK_{2, (X_+, D)} \to 0.
\]
Since $H^2(S_X, \iota_{*}(\sK_{2, (S_X, X_-, D)}) = 
H^2(D, \sK_{2, (S_X, X_-, D)}) = 0$,
we get 
$H^1(X, \sK_{2, (S_X, X_-)}) \surj H^1(X, \sK_{2, (X_+,D)}) \simeq
H^1(X, \sK_{2, (X,D)})$.
\end{proof}

\section{Roitman's torsion theorem for 0-cycles with 
modulus}\label{sec:RTM}
In this section, we prove the Roitman torsion theorem 
(\thmref{thm:Intro-2}) for the Chow group of 0-cycles with modulus.
We shall deduce this result by proving a more general Roitman torsion theorem 
for singular varieties which are obtained by joining two smooth schemes
along a common reduced Cartier divisor. 
As before, we assume the base field $k$ to 
be algebraically closed of exponential characteristic $p$.

\enlargethispage{25pt}

\subsection{Join of two smooth schemes along a common divisor}
\label{sec:gluing}
Let $X_+$ and $X_-$ be two smooth connected quasi-projective schemes of 
dimension $d \ge 1$ over $k$ and let 
$X_+ \stackrel{i_+}{\hookleftarrow} D \stackrel{i_-}{\hookrightarrow} X_-$
be two closed embeddings such that $D$ is a reduced effective Cartier divisor
on each $X_\pm$ via these embeddings. Let $X$ be the quasi-projective scheme
over $k$ such that the square
\begin{equation}\label{eqn:Bertini-0}
\xymatrix@C1pc{
D \ar[r]^-{i_+} \ar[d]_{i_-} & X_+ \ar[d]^{\iota_+} \\
X_- \ar[r]_-{\iota_-} & X}
\end{equation}
is co-Cartesian in $\Sch_k$.

It is easy to check that all arrows in ~\eqref{eqn:Bertini-0} are closed 
immersions and $X$ is a reduced scheme with irreducible
components $\{X_+, X_-\}$ with $X_{\rm sing} = D$. In particular, the
canonical map $\pi = \iota_+ \amalg \iota_-: X_+ \amalg X_- \to X$ is the
normalization map. 
Let $U = X_{\rm reg} = (X_+ \setminus D) \amalg (X_- \setminus D)$.
We shall use the following important further properties of $X$.

\begin{lem}\label{lem:CM-join}
The scheme $X$ satisfies the following properties.
\begin{enumerate}
\item
It is Cohen-Macaulay.
\item
It is separably weakly normal.
\item
It is weakly normal.
\item
The map $D \to X_+ \times_X X_-$ is an isomorphism.
\end{enumerate}
\end{lem}
\begin{proof}
Because $X_+$ and $X_-$ are smooth of dimension $d$, they are
Cohen-Macaulay. Since $D$ is an effective Cartier divisor on
a smooth scheme of dimension $d$, it is Cohen-Macaulay of dimension $d-1$.
The statement (1) now follows from 
\cite[Lemma~1.2, Lemma~1.5, (1.5.3)]{AAM}. 
Second statement follows because $D \subset X$ is the smallest conducting
closed subscheme for the normalization $\pi: X^N \to X$ which is reduced.
Furthermore, the map $D \times_X X^N \simeq D \amalg D \xrightarrow{\pi} D$ is 
just the collapse map and hence generically separable.
The statement (3) follows from the previous two and  (4)
follows from \lemref{lem:rel-et-2*}.
\end{proof}

\subsection{The main result}\label{sec:Main-RT}
We shall use the set-up of \S~\ref{sec:gluing} throughout this section. 
We assume from now on that $d = \dim(X) \ge 3$.
Let $C \inj X$ be a reduced Cartier curve (see \S~\ref{sec:Chow-sing})
such that each irreducible
component of $C \setminus D$ is smooth and $C \setminus D$ has only double point
singularities.

\begin{lem}\label{lem:Bertini}
There exists a locally closed embedding $X \inj \P^n_k$ such that for a 
general set of distinct hypersurfaces $H_1, \cdots , H_{d-2} \subset \P^n_k$
of large degree containing $C$, the intersection
$L = H_1 \cap \cdots \cap H_{d-2}$ satisfies the following.
\begin{enumerate}
\item
$X \cap L$ is reduced.
\item
$D \cap L$ is reduced.
\item
$X_\pm \cap L$ is an integral normal surface 
whose singular locus is contained in $C_{\rm sing} \cap D$.
\item
The inclusion $C \subset (X \cap L)$ is a local complete intersection along $D$.
\end{enumerate}
\end{lem}
\begin{proof}
Since $X$ is reduced of dimension $d \ge 3$, 
and since $C \subset X$ is a local complete intersection along 
$D = X_{\rm sing}$, it follows from \cite[Lemmas~1.3, 1.4]{Levine-2} (see also
\cite[Sublemma~1]{BS-1}) that for all $m \gg 1$, there is
an open dense subset $\sU_n(C, m)$ of the scheme $\sH_n(C, m)$ 
of hypersurfaces in $\P^n_k$ of degree $m$ containing $C$ such that 
the following hold.
\begin{listabc}
\item
For general distinct $H_1, \cdots , H_{d-2} \in \sU_n(C, m)$,
the scheme-theoretic intersection $L = H_1 \cap \cdots \cap H_{d-2}$
has the property that $X \cap L$ is reduced away from 
\nolinebreak
$C$.
\item
$D \cap L$ is reduced away from $C$.
\item
$X_\pm \cap L$ is integral away from $C$.
\item
$C \subset (X \cap L)$ is a local complete intersection along $D$.
\end{listabc}

Note that as $X_\pm$ is smooth and $X_\pm \cap L$ is a complete
intersection in $X_\pm$, it follows that $X_\pm \cap L$ is a 
Cohen-Macaulay surface. In particular, it has no embedded component.
Since $C$ is nowhere dense in $X_\pm \cap L$, it follows from (c) that 
$X_\pm \cap L$ must be irreducible. Since $C$ does not contain the
generic point of $X_\pm \cap L$, it follows from (b) and 
\lemref{lem:red-finite} that $X_\pm \cap L$ is integral.

Setting $W:= C\setminus (C_{\rm sing} \cap D)$ and following \cite[\S~5]{KL}, 
let ${W}(\Omega^1_{C}, e)$ denote the locally closed subset of points in $W$ 
where the embedding dimension of $W$ is $e$. It follows from our
assumption on the singularities of $C \setminus D$ that 
${\underset{e}{\rm max}} \{\dim({W}(\Omega^1_{C}, e)) + e\} \le 2$.
We conclude from \cite[Theorem~7]{KL} that for all $m \gg 1$, there is an 
open dense subset $\sW_n(W, m)$ of the scheme $\sH_n(W, m)$
of hypersurfaces of $\P^n_k$ of degree $m$ containing $W$ such that  
for general distinct $H_1, \cdots , H_{d-2} \in \sW_n(W, m)$,
the scheme-theoretic intersection $L = H_1 \cap \cdots \cap H_{d-2}$
has the following properties.
\begin{listabcprime}
\item
$X \cap L$ is a complete intersection in $X$ of dimension two.
\item
$(X \cap L) \setminus D$ is smooth.
\item
$X_\pm \cap L$ is smooth away from $C_{\rm sing} \cap D$.
\end{listabcprime}

It follows from (c') that $X_\pm \cap L$ is a Cohen-Macaulay surface
whose singular locus is contained in $C_{\rm sing} \cap D$. It follows
from Serre's normality condition that $X_\pm \cap L$ is normal.
Since $C$ is the closure of $W$ in $X$, we must have 
$\sH_n(W, m) = \sH_n(C, m)$ and in particular, $D \subset X \cap L$.
Combining (a) - (d) and (a') - (c') together, we conclude that there is a 
closed immersion $Y \subset X$ with the following properties.
\begin{enumerate}
\item
$\dim(Y) = 2$.
\item
The inclusions $Y \subset X$, $(X_\pm \cap Y) \subset X_\pm$ and
$(D \cap Y) \subset D$ are all complete intersections.
\item
$Y$ is reduced away from $C \cap D$.
\item
$(D \cap Y) \subset Y$ is a Cartier divisor which is reduced away from $C$.
\item
$X_\pm \cap Y$ is an integral normal surface which is smooth away from
$C_{\rm sing} \cap D$. 
\item
$C \subset Y$ is a local complete intersection along $D$.
\item
$Y \setminus D$ is smooth.
\end{enumerate}

If $Y \subset X$ is as above, then (6) says that the local rings
of $Y$ at $C \cap D$ contain regular elements. 
In particular, $C \cap D$ can contain no embedded point of $Y$. We
conclude from (3) and \lemref{lem:red-finite} that a surface $Y$
as above must be reduced.
Since $D \subset X_\pm$ is a Cartier divisor, it is Cohen-Macaulay
and hence (2) shows that $D \cap Y$ is a complete intersection 
Cohen-Macaulay curve inside $D$. Since $C \cap D$ can contain no 
generic point of $D \cap Y$, it follows form (4) and \lemref{lem:red-finite}
that $D \cap Y$ is a reduced curve.
This finishes the proof of the lemma.
\end{proof}

The following is a straightforward application of the
prime avoidance theorem in commutative algebra and its proof
is left to the reader.

\begin{lem}\label{lem:red-finite} 
Let $R$ be commutative Noetherian ring such that $R_{\fp}$ is reduced 
for every associated prime ideal $\fp$ of $R$.
Then $R$ must be reduced.
\end{lem}

\begin{lem}\label{lem:Bertini-Roitman}
Let $X = X_+ \amalg_D X_-$ be as in ~\eqref{eqn:Bertini-0} and let 
$Y = X \cap L$ be the complete intersection surface as obtained in
\lemref{lem:Bertini}. Assume that $X$ is projective. Then, $A^2(Y)$ 
is a semi-abelian variety and the Roitman torsion theorem holds for $Y$. 
That is, the Abel-Jacobi map $\rho_Y: \CH^{LW}_0(Y)_{\deg 0} \to A^2(Y)$ 
is an isomorphism on the torsion subgroups.
\end{lem}
\begin{proof}
We let $Y_\pm = X_\pm \cap Y$ and $E = D \cap Y$.
Let $\iota'_\pm: Y_\pm \inj Y$ be the inclusion maps.
It follows from the construction of $Y \subset X$ and an easy variant of
\cite[Lemma~2.2]{BK} that the canonical map $Y_+ \amalg_E Y_- \to Y$
is an isomorphism. It also follows from \lemref{lem:Bertini} that
$\iota'_+ \amalg \iota'_-: Y_+ \amalg Y_- \to Y$ is the normalization map. 
Since $Y$ is the join of normal surfaces $Y_+$ and $Y_-$ 
along the common closed subscheme $E$, it follows that the non-normal
locus of $Y$ is the support of $E$. Let us denote the map 
$\iota'_+ \amalg \iota'_-$ by $\pi'$.

Let $\sI_\pm \subset \sO_{Y_\pm}$ denote the defining ideal sheaf for the
inclusion $E \subset Y_\pm$. It is then immediate from ~\eqref{eqn:Bertini-0} 
that $\pi'^*(\sI_E) = \sI_+ \times \sI_-$ which is actually in $\sO_Y$
under the inclusion $\pi'^*: \sO_Y \inj \pi'_*(\sO_{Y_+ \amalg Y_-})$.
In other words, $E \inj Y$ is a conducting subscheme for $\pi'$.
Since $\pi'^*(E) = E \amalg E$ is reduced by \lemref{lem:Bertini}
and since the support of $E$ is the non-normal locus of $Y$, we
see that $\pi'$ is the normalization map for which $E$ is the smallest 
conducting subscheme and is reduced. 
Since $\pi': E \amalg E \to E$ is just the collapse map,
it is clearly generically separable. We conclude 
(see Definition~\ref{defn:swn}) that $Y$ is a reduced projective surface
which is separably weakly normal. The lemma now follows
from \thmref{thm:Fin-dim-main}.
\end{proof}

\begin{thm}\label{thm:RT-join}
Let $X_+$ and $X_-$ be two smooth projective schemes of dimension $d \ge 1$
over $k$ and let $X = X_+ \amalg_D X_-$ be as in ~\eqref{eqn:Bertini-0}.
Then, the Albanese variety $A^d(X)$ is a semi-abelian variety
and the Abel-Jacobi map $\rho_{X}: \CH^{LW}_0(X)_{\deg 0} \to A^d(X)$ is
isomorphism on the torsion subgroups.
\end{thm}
\begin{proof}
It follows from \lemref{lem:CM-join} that
$X$ is separably weakly normal and weakly normal.  Hence, in characteristic 
zero, a general reduced Cartier curve on $X$ containing a
0-cycle is weakly normal by \cite[Corollary~2.5]{CGM}.
We can thus repeat the argument of the proof of \thmref{thm:Fin-dim-main},
which shows that $\rho_X$ is surjective and reduces the remaining proof to 
showing only that the restriction of $\rho_{X}$ to the
$p$-primary torsion subgroup of $\CH^{LW}_0(X)_{\deg 0}$ 
is injective when $p \ge 2$.

When $d=1$, the scheme $X$ is a weakly normal curve, and it is well known in 
this case that $\Pic^0(X) \simeq A^1(X) \simeq J^1(X)$.
The case $d = 2$ is \lemref{lem:Bertini-Roitman}.
So we assume $d \ge 3$.

Let $\alpha \in \sZ_0(X)$ be such that $n\alpha = 0$ in
$\CH^{LW}_0(X)$ for some integer $n = p^m$. Equivalently,
$n\alpha \in \sR^{LW}_0(X)$. Let us assume further 
that $\rho^{\rm semi}_X(\alpha) = 0$.

We can now use \cite[Lemma~2.1]{BS-1} to find a reduced Cartier curve
$C$ on $X$ and a function $f \in \sO^{\times}_{C, S}$ such that $n\alpha = 
\divf(f)$, where $S = C \cap X_{\rm sing} = C \cap D$.
Since part of $n\alpha$ supported on any connected component of
$C$ is also of the form
$n\alpha' = \divf_{C'}(f')$ for some Cartier curve $C'$ and some $f'$, 
we can assume that $C$ is connected. Let $\{C_1, \cdots , C_r\}$ denote
the set of irreducible components of $C$.

We set $U = X \setminus D = (X_+ \setminus D) \amalg (X_- \setminus D)$.
Let $\phi:X' \to X$ be a successive blow-up at smooth points such that the
following hold.
\begin{enumerate}
\item
The strict transform $D_i$ of each $C_i$ is smooth along $\phi^{-1}(U)$.
\item
$D_i \cap D_j \cap \phi^{-1}(U) = \emptyset$ for $i \neq j$.
\item
Each $D_i$ intersects the exceptional divisor $E$ (which is reduced) 
transversely at smooth points.
\end{enumerate}

It is clear that there exists a finite set of blown-up closed points
$T = \{x_1, \cdots , x_n\} \subset U$ such that 
$\phi: \phi^{-1}(X\setminus T) \to X \setminus T$
is an isomorphism. In particular, $\phi: X'_{\rm sing} = \phi^{-1}(D) \to D$
is an isomorphism.  If we identify $\phi^{-1}(D)$ with $D$ and let
$X'_\pm = {\rm Bl}_{T \cap X_\pm}(X_\pm)$, it becomes clear that
$X' = X'_+ \amalg_D X'_-$. We set $U' = 
X' \setminus D = \phi^{-1}(U)$. 
Let $C'$ denote the strict transform of $C$ with
components $\{C'_1, \cdots, C'_r\}$.

Since $\phi$ is an isomorphism over an open neighborhood of $D$,
it follows that $\phi^{-1}(S) \simeq S$ and the map $C' \to C$ is an
isomorphism along $D$. In particular, we have
$f \in \sO^{\times}_{C', S}$ and $\phi_*(\divf_{C'}(f)) = \divf_C(f) = n\alpha$.
Since ${\rm Supp}(\alpha) \subset C'$, we can find
$\alpha' \in \sZ_0(X')$ supported on $C'$ 
such that $\phi_*(\alpha') = \alpha$.
This implies that 
$\phi_*(n\alpha' - \divf_{C'}(f)) = 0$.
Setting $\beta = n\alpha' - \divf_{C'}(f)$,
it follows that $\beta$ must be a 0-cycle on the exceptional divisor
such that $\phi_*(\beta) = 0$.

We can now write $\beta = \stackrel{n}{\underset{i =0}\sum} \beta_i$,
where $\beta_i$ is a 0-cycle on $X'$ supported on $\phi^{-1}(x_i)$
for $1 \le i \le n$ and $\beta_0$ is supported on the complement $E$. 
We must then have $\phi_*(\beta_i) = 0$
for all $0 \le i \le n$. Since $\phi$ is an isomorphism away from 
$T = \{x_1, \cdots , x_n\}$, we must have $\beta_0 = 0$.
We can therefore assume that $\beta$ is a 0-cycle on $E$.

Next, we note that each $\phi^{-1}(\{x_i\})$ is a $(d-1)$-dimensional 
projective 
variety whose irreducible components are point blow-ups of $\P^{d-1}_k$,
intersecting transversely in $X'_{\rm reg}$. Moreover, we have 
$\phi_*(\beta_i) = 0$ for the push-forward map 
$\phi_*: \sZ_0(\pi^{-1}(\{x_i\})) \to \Z$, induced by $\phi: \phi^{-1}(\{x_i\})
\to \Spec(k(x_i)) = k$. But this means that 
${\rm deg}(\beta) = \stackrel{n}{\underset{i =0}\sum} {\rm deg}(\beta_i) = 0$.
In particular, there are finitely many smooth projective rational curves 
$L_j \subset E$ and rational functions $g_j \in k(L_j)$ such that
$\beta = \stackrel{s}{\underset{j = 1}\sum}  (g_j)_{L_j}$.

Using the argument of \cite[Lemma~5.2]{Bloch-duke}, 
we can further choose $L_j$'s so that
$C'':= C' \cup (\cup_j L_j)$ is a connected reduced curve with
following properties.
\begin{enumerate}
\item
Each component of $C''$ is smooth along $U'$.
\item
$C'' \cap U'$ has only ordinary double point singularities. 
\end{enumerate}

In particular, the embedding 
dimension of $C''$ at each of its singular points lying over $U$ is two.
Furthermore, $C'' \cap D = 
(C'' \setminus (\cup_j L_j)) \cap D = C' \cap D$.
This implies that $C''$ is a Cartier curve on $X'$.

We now fix a closed embedding $X'  = X'_+ \amalg_D X'_- \inj \P^N_k$
and choose a complete intersection surface $j':Y' \subset X'$ 
as in \lemref{lem:Bertini}. Let $\pi': X'_+ \amalg X'_- \to X'$
denote the normalization map with $\pi'^{-1}(Y') = Y'_+ \amalg Y'_-$.
Note that $E' = Y' \cap D$
is a reduced curve and $Y'_{\rm sing} = E'$. 
Since $C'' \cap E' = C' \cap E'$ and since $C''$ is Cartier on $Y'$, 
it follows that $C'$ and $L_j$'s are also Cartier curves on $Y'$.
Furthermore, $\alpha'$ is an element of
$\sZ_0(Y', E')$ such that $n\alpha' = \divf_{C'}(f) + \sum_j \divf_{L_j}(g_j)$.
In particular, $\alpha' \in \CH^{LW}_0(Y')$ and $n\alpha' = 0$ in 
$\CH^{LW}_0(Y')$.
Note that this also implies that $\alpha' \in \CH^{LW}_0(Y')_{\deg 0}$. 

It follows from \cite[Theorem~14]{Mallick} that there exists a commutative
diagram
\begin{equation}\label{eqn::RT-join-0}
\xymatrix@C1pc{
\CH^{LW}_0(Y')_{\deg 0} \ar[d]_{j'_*} \ar[r]^-{\rho^{\rm semi}_{Y'}} &
J^2(Y') \ar[d]^{j'_*} \ar[r]^-{\pi'^*} & 
J^2(Y'_+) \times J^2(Y'_-) \ar[d]^{j'_*} \\ 
\CH^{LW}_0(X')_{\deg 0} \ar[r]_-{\rho^{\rm semi}_{X'}} & 
J^d(X') \ar[r]_-{\pi'^*} & J^d(X'_+) \times J^d(X'_-).}
\end{equation}

Since $\phi: X' \to X$ is a blow-up at smooth closed points and since
the Albanese variety of a smooth projective scheme is a birational 
invariant, it follows from the construction of $J^d(X)$ in
\S~\ref{sec:SAQ} and \cite[Lemma~2.5]{ESV} that the canonical push-forward map 
$\phi_*: \sZ_0(X') \to \sZ_0(X)$ gives rise to a commutative diagram

\begin{equation}\label{eqn::RT-join-1}
\xymatrix@C1pc{
\CH^{LW}_0(X')_{\deg 0} \ar[r]^-{\rho^{\rm semi}_{X'}} \ar[d]_{\phi_*} & 
J^d(X') \ar[d]^{\phi_*} \\
\CH^{LW}_0(X)_{\deg 0} \ar[r]_-{\rho^{\rm semi}_{X}} &
J^d(X),}
\end{equation}
in which the two vertical arrows are isomorphisms.
It follows that $\alpha' \in \CH^{LW}_0(Y')_{\deg 0}$ is a 0-cycle such that
$\rho^{\rm semi}_{X'} \circ j'_*(\alpha') = 0$ in $J^d(X')$.
Using ~\eqref{eqn::RT-join-0}, we conclude that
\begin{equation}\label{eqn::RT-join-2}
j'_* \circ \pi'^* \circ \rho^{\rm semi}_{Y'}(\alpha') =
\pi'^* \circ \rho^{\rm semi}_{X'} \circ j'_*(\alpha') = 0.
\end{equation}

Since $Y'_+ \amalg Y'_- \to Y'$ is the normalization map
(see \lemref{lem:Bertini-Roitman}), we know that
$J^2(Y'_+) \times J^2(Y'_-)$ is the universal abelian variety 
quotient of $J^2(Y')$.
In particular, the map $J^2(Y')\{p\} \to J^2(Y'_+)\{p\} \times J^2(Y'_-)\{p\}$
is an isomorphism. On the other hand, $Y'_\pm \subset X'_\pm$ are the
iterated hypersurface sections of normal projective schemes and hence
it follows from \cite[Chap.~8, \S~2, Theorem~5]{Lang} that
the right vertical arrow in ~\eqref{eqn::RT-join-0} is an isomorphism of 
abelian varieties.
Note here that $X'$ or $Y'$ need not be smooth for this isomorphism.
It follows therefore from  ~\eqref{eqn::RT-join-2} that
$\rho^{\rm semi}_{Y'}(\alpha') = 0$. 
\lemref{lem:Bertini-Roitman} now implies that $\alpha' = 0$ and 
we finally get $\alpha = \phi_*(\alpha') = 0$.
\end{proof}

\enlargethispage{25pt}

\subsection{Applications of \thmref{thm:RT-join}}\label{sec:Appl-RT-join}
We now obtain some applications of \thmref{thm:RT-join}.
Our first result is the following comparison theorem.

\begin{thm}\label{cor:LW-lci-join}
Let $X_+$ and $X_-$ be two smooth projective schemes of dimension $d \ge 1$
over $k$ and let $X = X_+ \amalg_D X_-$ be as in ~\eqref{eqn:Bertini-0}.
Then, the canonical map $\CH^{LW}_0(X) \to \CH_0(X)$ is an isomorphism.
\end{thm}
\begin{proof}
Let $L$ denote the kernel of the surjection $\CH^{LW}_0(X) \surj \CH_0(X)$.
Using the factorization $\CH^{LW}_0(X) \surj \CH_0(X) \to K_0(X)$
(see \cite[Lemma~3.13]{BK} and \cite[Corollary~2.7]{Levine-2}),
we know that $L$ is a torsion subgroup of $\CH^{LW}_0(X)_{\deg 0}$ of bounded
exponent. On the other hand, we also have a factorization
$\CH^{LW}_0(X)_{\deg 0} \surj \CH_0(X)_{\deg 0} \surj J^d(X)$ by
~\eqref{eqn:alb-lci}. Since $J^d(X)$ is the universal semi-abelian variety
quotient of $A^d(X)$, we can use \thmref{thm:RT-join} to replace $J^d(X)$ by
$A^d(X)$. We therefore have a factorization $\CH^{LW}_0(X)_{\deg 0} \surj 
\CH_0(X)_{\deg 0} \surj A^d(X)$  of $\rho^{\rm semi}_X$. Another application of
\thmref{thm:RT-join} now shows that $L$ is torsion-free. Hence, it must be zero.
\end{proof}

As second application of \thmref{thm:RT-join}, we now prove the Roitman 
torsion theorem for the Chow group of 0-cycles 
with modulus when the underlying divisor is reduced.

\begin{thm}\label{thm:RT-Main}
Let $X$ be a smooth projective scheme of dimension $d \ge 1$ over $k$ and let
$D \subset X$ be an effective Cartier divisor. Assume that $D$ is reduced.
Then, the Albanese variety with modulus $A^d(X|D)$ is a semi-abelian variety
and the Abel-Jacobi map $\rho_{X|D}: \CH_0(X|D)_{\deg 0} \to A^d(X|D)$ is
isomorphism on the torsion subgroups.
\end{thm}
\begin{proof}
By \thmref{cor:LW-lci-join} and \cite[(11.2)]{BK}, 
there is a commutative diagram of split exact sequences:
\[
\xymatrix@C.8pc{ 
0\to \CH_0(X|D)_{\deg 0 } \ar[r]^-{p_{+,*}} 
\ar[d]_{\rho_{X|D}} & \CH^{LW}_0(S_X)_{\deg 0} \ar[r]^-{i_-^*} 
\ar[d]^{\rho_{S_X}} & \CH_0(X)_{\deg 0} \ar[d]^{\rho_X}  \to 0\\ 
0\to A^d(X|D) \ar[r]_-{p_{+,*}}  & A^d(S_X) \ar[r]_-{i_-^*} 
& A^d(X) \to 0.}
\]

We note here that the constructions of \cite[\S~11]{BK} are based on
the assumption that \thmref{cor:LW-lci-join} holds.
Since $X$ is smooth over $k$, we know that $A^d(X)$ is an abelian variety.
Moreover, the right vertical arrow is isomorphism on 
the torsion subgroups by Roitman \cite{Roitman}
and Milne \cite{Milne-1}. 
We conclude that the theorem is equivalent to showing that
$A^d(S_X)$ is a semi-abelian variety and $\rho_{S_X}$ is isomorphism on 
the torsion subgroups. But this follows from 
Theorem~\ref{thm:RT-join}.
\end{proof}

\begin{cor}\label{cor:fin-dim-cor-mod}
Let $X$ be a smooth projective scheme of dimension $d \ge 1$ over $\ov{\F}_p$ 
and let $D \subset X$ be a reduced effective Cartier divisor. 
Then, $\CH_0(X|D)_{\deg 0}$ is finite-dimensional. That is, the Abel-Jacobi map
$\rho_X: \CH_0(X|D)_{\deg 0} \to A^d(X|D)$ is an isomorphism.
\end{cor}
\begin{proof}
In view of \thmref{thm:RT-Main}, we only need to show that 
$\CH_0(S_X)$ is a torsion abelian group. But this is already shown in
the proof of \corref{cor:fin-dim-cor}.
\end{proof}

\vskip .3cm


\noindent\emph{Acknowledgments.} 
The author would like to thank the referee for very carefully reading the
paper and providing many suggestions to improve its presentation.

\enlargethispage{15pt}

\end{document}